\documentclass[reqno,12pt,letterpaper]{amsart}
\usepackage{amsmath,amssymb,amsthm,graphicx,mathrsfs,url}
\usepackage[usenames,dvipsnames]{color}
\usepackage[colorlinks=true,linkcolor=Red,citecolor=Green]{hyperref}

\def\?[#1]{\textbf{[#1]}\marginpar{\Large{\textbf{??}}}}

\newcommand{\RR}{{\mathbb R}}

\renewcommand{\Re}{\mathop{\rm Re}\nolimits}
\renewcommand{\Im}{\mathop{\rm Im}\nolimits}

\DeclareMathOperator{\Op}{Op}
\DeclareMathOperator{\supp}{supp}
\DeclareMathOperator{\Tr}{Tr}
\DeclareMathOperator{\Vol}{Vol}

\newtheorem{thm}{Theorem}
\newtheorem{prop}{Proposition}%[section]

\newtheorem{lem}[prop]{Lemma}
\newtheorem{cor}[prop]{Corollary}
\newtheorem{rem}{Remark}
\numberwithin{equation}{section}
\numberwithin{prop}{section}

\begin{document}
\title{Resonance-free Region in scattering by a strictly convex obstacle}
\author{Long Jin}
\email{jinlong@math.berkeley.edu}
%\address{Department of Mathematics, Evans Hall, University of California,
%Berkeley, CA 94720, USA}
\address{Department of Mathematics, Evans Hall, University of California,
Berkeley, CA 94720, USA}
\date{}
\maketitle

\begin{abstract}
We prove the existence of a resonance free region in scattering by a strictly convex obstacle $\mathcal{O}$ with the Robin boundary condition $\partial_\nu u+\gamma u|_{\partial\mathcal{O}}=0$. More precisely, we show that the scattering resonances lie below a cubic curve $\Im\zeta=-S|\zeta|^{\frac{1}{3}}+C$. The constant $S$ is the same as in the case of the Neumann boundary condition $\gamma=0$. This generalizes earlier results on cubic poles free regions \cite{BLR}, \cite{HL} ,\cite{SZ5} obtained for the Dirichlet boundary condition.
\end{abstract}

\section{Introduction and Statement of Results}
\label{int}

Let $\mathcal{O}$ be a bounded open set in $\RR^n$ with a smooth boundary. For $\gamma$, a real-valued smooth function on $\partial\mathcal{O}$,
we define  $P^{(\gamma)} $
%:D^{(\gamma)}(\RR^n\setminus\mathcal{O})\to L^2(\mathbb{R}^n\setminus\mathcal{O})$
to be the Laplacian on the exterior domain $-\Delta_{\mathbb{R}^n\setminus\mathcal{O}}$ realized with the Robin boundary condition $\partial_\nu u+\gamma u|_{\partial\mathcal{O}}=0$.
%Here the domain $D^{(\gamma)}(\RR^n\setminus\mathcal{O})$ is chosen to be
%$\{u\in H^2(\RR^n\setminus\mathcal{O}):\partial_\nu u+\gamma u|_{\partial\mathcal{O}}=0$,(in the sense of the trace)$\}$, so that
The operator $P^{(\gamma)}$ is a self-adjoint operator with continuous spectrum $[0,\infty)$. The resolvent
and $  R^{(\gamma)}(\zeta) :=(P^{(\gamma)}-\zeta^2)^{-1} \; : \; L^2 ( \RR^n \setminus \mathcal{O} )\to L^2 ( \RR^n \setminus \mathcal{O} )$
is holomorphic in $\Im\zeta>0$. The Green function, $ G^{(\gamma)} ( \zeta , x, y ) $ is defined as the
kernel of this resolvent:
\begin{gather*}
u ( x ) = \int_{ \RR^n \setminus \mathcal{O} } G^{(\gamma)} ( \zeta, x , y ) f ( y ) dy , \ \ f \in C^\infty_{\rm{c}} ( \RR^n
\setminus \mathcal{O} ),
\\
( - \Delta - \zeta^2 ) u ( x ) = f ( x ) , \ \ x \in  \RR^n \setminus \bar {\mathcal{O} } , \ \
\partial_\nu u ( x ) +\gamma ( x ) u ( x ) = 0 , \ \ x \in \partial \mathcal O.
\end{gather*}
The Green function has a meromorphic extension in $ \zeta $ to the whole complex plane if $n$ is odd, and the logarithmic covering space of $\mathbb{C}\setminus\{0\}$ if $n$ is even.

The scattering poles or resonances of $P^{(\gamma)}$ are defined as the poles of this meromorphic continuation and they enjoy interesting interpretations and properties. The real part of a resonance corresponds to a frequency of a resonant wave,
and the imaginary part, to the decay rate of the wave. Consequently understanding of the separation of resonances
from the real axis is related to the decay properties and long time behavior of waves.

\begin{figure}
\label{Fig:pfr}
\centering
\includegraphics[scale=0.7]{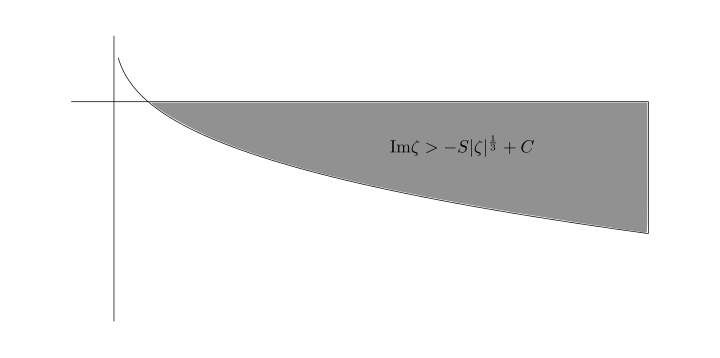}
\caption{The Resonance-Free Region}
\end{figure}

Resonance free regions near the real axis have been extensively studied since the work
of Lax-Phillips \cite{LP} and Vainberg \cite{V} where the presence of such regions
was linked to propagation of singularities for the wave equation, and hence to the
geometry of the obstacle $ {\mathcal O}$.  Thanks to the work of Melrose, Ivrii, Sj\"ostrand and Taylor
on propagation of singularities for boundary value problems, if $ {\mathcal O} $ is {\em nontrapping},
that is, all reflecting rays escape to infinity,  there are no resonances in the region
$ \Im \zeta > - M \log | \zeta | + C_M $ for any $ M $ -- see \cite[Chapter 24]{Ho},\cite{Me},\cite{TZ} and references given there.
When the boundary $ \partial \mathcal O $ is {\em real analytic}, and the obstacle is {\em nontrapping},
the work of Lebeau \cite{Le} on propagation of Gevrey 3 singularities implies that the resonance free
region is cubic, that is, there are no resonances in $ \Im \zeta > - C_0 | \zeta|^{1/3}  + C_1 $ for
some constants $ C_0,C_1 $ -- see Popov \cite{Po} and Bardos-Lebeau-Rauch \cite{BLR}.
 The above results have been typically stated in the case of the Dirichlet
boundary condition, as they depend only on propagation of singularities, they hold for $ P^{(\gamma)}$,
with $ \gamma $ analytic, for the cubic poles free region.

The case of stictly convex obstacles has been studied, first when $ {\mathcal O} $ is the sphere, since
the work of Watson \cite{W} on electromagnetic scattering by the earth almost a hundred years ago.
In that case, for the Dirichlet or Neumann problems scattering poles are given by zeros of Hankel
function or their derivatives, respectively -- see Stefanov \cite{St} for a modern account and references.
Since a convex obstacle is nontrapping the existence of a logarithmic resonance free region
follows from the general results \cite{Me}. It was first established, in the star shaped case, by Morawetz-Ralston-Strauss \cite{MRS}.

A remarkable discovery was made by Harg\'e-Lebeau \cite{HL} who showed that for a smooth
boundary and the Dirichlet boundary condition,
the resonance free region is cubic, that is, the same result is valid as for analytic
non-trapping obstacles. This result is sharp as was shown already in \cite{BLR} where the analysis
of Gevrey-3 singularities of the wave trace gave a string of of resonances near a cubic curve
-- see also \cite{SjC} and \cite[Example 4]{SZ3}. The argument of Harg\'e-Lebeau was based on the
complex scaling method of Aguilar-Combes \cite{AC} and Balslev-Combes \cite{BC} developed for boundary value problems by
Sj\"ostrand-Zworski \cite{SZ1},\cite{SZ4}.  In \cite{SZ5}
Sj\"ostrand-Zworski gave a more direct proof of the cubic resonance free region for smooth strictly  convex obstacles
and established polynomial bounds on the number of resonances in neighborhood of the real axis. In \cite{SZ6}
they proved an asymptotic law for the number of resonances in the cubic strips for obstacles whose boundary
satisfies a pinched curvature assumption.

In this paper, we prove that the results of  \cite{HL} and \cite{SZ5} on the cubic resonance free region
are valid for an arbitrary Robin boundary condition. The main result
which follows from Theorem \ref{lb:sc} below is
\begin{thm}
\label{polefree}
Suppose that $\mathcal{O}$ is strictly convex with a smooth boundary. Then there are no resonances of $P^{(\gamma)}$ in the cubic region
\begin{equation}
\Im\zeta>-S|\zeta|^{\frac{1}{3}}+C,
\end{equation}
for some $C>0$, depending on $\gamma$ and $\mathcal{O}$. The constant $S$ is given by
\begin{equation}
S=2^{-\frac{1}{3}}\cos\left(\frac{\pi}{6}\right)\zeta_1'\left(\min_{y\in\mathcal{O},i=1,\ldots,n-1}K_i(y)\right)^{\frac{2}{3}}
\end{equation}
with $K_i(y)$ the principal curvatures of $\partial\mathcal{O}$ at $y$, $-\zeta_1'$ the first zero of the derivative of the Airy function.
\end{thm}

\begin{rem}
The constant $S$ here is optimal if the obstacle $\mathcal{O}$ is a ball and $\gamma$ is a constant, since we can get explicit expressions for the resonances by Hankel functions -- see \cite{St}. When $ \gamma $ is not a constant function Theorem \ref{polefree} is new even in the case of the sphere.

In the case of Dirichlet or Neumann boundary condition, the optimal value for the constant $S$ is given in \cite{BLR} and \cite{SjC} when $\partial\mathcal{O}$ is analytic and its geodesics satisfy certain conditions. The optimal constant $S$ in other cases is still unknown.
\end{rem}

The basic strategy is similar to that in \cite{SZ5} but with some challenges provided by the more
general boundary condition. First, we write the operator in normal geodesic coordinates with respect to the boundary and apply the method of complex scaling all the way to the boundary. For this part, we shall follow the approach in \cite{SZ1} and \cite{SZ4} and reduce the problem of resonances free region to the problem of localizing the spectrum of a non-selfadjoint operator.
Then we shall use the global FBI transform on the boundary \cite{WZ} to change the complex-scaled Laplace operator to an Airy-type ordinary differential operator, at least near the boundary. In Section \ref{airy}, we study the model Airy-type ordinary differential operator with general boundary condition. In Section \ref{lb}, we will go back to the scaled operator and prove a certain lower bound
for it -- see Theorem \ref{lb:sc}. Finally, in Section \ref{pf}, we use that lower bound to show the existence of a resonance free region stated in the main theorem.
In the Appendix we review the complex scaling up the boundary \cite{SZ4} as additional care is needed when dealing with general boundary conditions.

\subsection*{\normalsize Acknowledgement}
I would like to thank Maciej Zworski for encouragement and  advice during the preparation of this paper.
Partial support by the National Science Foundation grant DMS-1201417 is also gratefully acknowledged.

\section{Preliminaries}
\label{pr}
\subsection{The complex scaling method}
\label{cs}
In this section, we reduce the problem of resonances to the spectrum of a non-selfadjoint operator by the complex scaling method. This method was first introduced by Aguilar-Combes \cite{AC} and Balslev-Combes \cite{BC} in studying the continuous spectrum of Schr\"{o}dinger operators and later proved to be a strong tool in the study of resonances. We shall follow the methods of Sj\"{o}strand and Zworski and apply it to the case of Robin boundary condition. For more details, see \cite{SZ1}, \cite{SZ4} and \cite{SZ5}.

Let $\mathcal{O}$ be a strictly convex, bounded open set in $\RR^n$ with smooth boundary, then $d(x)=\mbox{dist}(x,\mathcal{O})$ is a smooth convex function in $\RR^n\setminus\mathcal{O}$. Moreover, $d''(x)\geqslant0$ and $\dim\ker d''(x)=1$, generated by $x-y(x)$ where $y(x)$ is the closest point to $x$ on $\partial\mathcal{O}$, so that $d(x)=|x-y(x)|$. At $y(x)$, the exterior unit normal vector of $\partial\mathcal{O}$ is $\nu(y(x))=\nabla d(y(x))=\frac{x-y(x)}{|x-y(x)|}$.

Near every $x_0\in\partial\mathcal{O}$, we shall choose the normal geodesic coordinates $(x',x_n),x'\in U\subseteq\RR^{n-1},x_n\geqslant0$ for $\RR^n\setminus\mathcal{O}$ where $(U,x'=(x_1,\ldots,x_{n-1}))$ is a local coordinates for $\partial\mathcal{O}$ centered at $x_0$, $x_n=d(x)$. This normal geodesic coordinates is given by $x=s(x')+x_n\nu(s(x'))$ where $s:U\to\partial\mathcal{O}$ is the coordinate map for $\partial\mathcal{O}$. In the global version, the normal geodesic coordinates identify $\RR^n\setminus\mathcal{O}$ with $\partial\mathcal{O}\times[0,\infty)$ by
\begin{equation*}
\RR^n\setminus\mathcal{O}\ni x\leftrightarrow(y,x_n)\in\partial\mathcal{O}\times[0,\infty)
\end{equation*}
where $y=s(x')\in\partial\mathcal{O}$.

Under this coordinates, for $x_n$ small, we have the following expression for $-h^2\Delta$:
\begin{equation*}
-h^2\Delta=(hD_{x_n})^2+R(x',hD_{x'})-2x_nQ(x',hD_{x'})+O(x_n^2(hD_{x'})^2)+O(h)hD_x+O(h^2)
\end{equation*}
where we introduce the semiclassical parameter $h>0$ which we will later let tend to zero, $R$ and $Q$ are two quadratic forms which are dual to the first and second fundamental form, respectively. Thus the principal curvatures of $\partial\mathcal{O}$ at $y=s(x')$ are the eigenvalues of $Q$ with respect to $R$.

Now following \cite{SZ4} and \cite{SZ5} we introduce the family of complex scaling contours $(0<\theta<\theta_0)$
\begin{equation*}
\Gamma_\theta=\{z=x+i\theta f'(x):x\in\RR^n\setminus\mathcal{O}\}\subset\RR^n\setminus\mathcal{O}+i\RR^n
\end{equation*}
where $f:\RR^n\setminus\mathcal{O}\to\RR$ is a smooth function such that for $x$ near $\mathcal{O}$, $f(x)=\frac{1}{2}d(x)^2$, so $f'(x)=d(x)d'(x)$.

If we parametrize $\Gamma_\theta$ by $x\in\RR^n\setminus\mathcal{O}$, then we can compute the principal symbol of $-\Delta|_{\Gamma_\theta}$:
\begin{equation*}
p_\theta(x,\xi)=\langle(1+i\theta f''(x))^{-1})\xi,(1+i\theta f''(x))^{-1}\xi\rangle=a_\theta-ib_\theta
\end{equation*}
where
\begin{equation*}
a_\theta=\langle(1-(\theta f''(x))^2\tilde{\xi},\tilde{\xi}\rangle, \ \ b_\theta=2\theta\langle f''(x)\tilde{\xi},\tilde{\xi}\rangle. \ \
\tilde{\xi}=(1+(\theta f''(x))^2)^{-1}\xi .
\end{equation*}
In normal geodesic coordinates, we can write the contours $\Gamma_\theta$ as the image of
\begin{equation*}
U\times[0,\infty)\ni(x',x_n)\mapsto s(x')+(1+i\theta)x_n\nu(s(x'))\in\mathbb{C}^n ,
\end{equation*}
locally for $x_n$ small. In the global version, by rescaling $t=|(1+i\theta)|x_n$, the contours $\Gamma_\theta$ are the image of
\begin{equation*}
\partial\mathcal{O}\times[0,\infty)\ni(y,t)\mapsto y+g_\theta(t)\nu(y)\in\mathbb{C}^n ,
\end{equation*}
where $g_\theta:[0,\infty)\to\mathbb{C}$ is a smooth injective map such that $|g_\theta'|=1,g(0)=0,g(t)=t\frac{1+i\theta}{|1+i\theta|}$ for $t$ near 0; $g(t)=t\frac{1+i\varphi(\theta)}{|1+i\varphi(\theta)|}$ outside a small neighborhood of 0 and
\begin{equation*}
\arg(1+i\varphi(\theta))\leqslant\arg g(t)\leqslant\arg(1+i\theta), \ \ \ \frac{1}{2}\arg(1+i\varphi(\theta))\leqslant\arg g'(t)\leqslant(1+i\theta) . \end{equation*}
We shall choose $\varphi(\theta)$ small enough such that $\Gamma_\theta$ satisfy the conditions given in \cite{SZ4}. In particular, we shall work on the contour $\Gamma=\Gamma_\theta$ for $\theta=\theta_1$ such that $g=g_{\theta_1}(t)$ equals to $te^{\frac{\pi i}{3}}$ for $t$ near 0 (later on we shall use $t\in[0,L^{-1}]$ for $L$ large enough). On this contour,
\begin{equation*}
\begin{split}
-h^2\Delta|_\Gamma & =\frac{1}{g'(t)^2}(hD_t)^2+R(y,hD_y)-2g(t)Q(y,hD_y)
\\ & \ \ \ \ \ \ \ +O(t^2(hD_y)^2)+O(h)hD_{y,t}+O(h^2) ,
\end{split}
\end{equation*}
which is elliptic in both semiclassical sense and the usual sense.

For $t$ small such that $g(t)=te^{\frac{\pi i}{3}}$,
\begin{equation*}
\begin{split}
-h^2\Delta|_\Gamma& =e^{-\frac{2\pi i}{3}}((hD_t)^2+2tQ(y,hD_y))+R(y,hD_y)\\
& \ \ \ \ \ \ \ \ +O(t^2(hD_y)^2)+O(h)hD_{y,t}+O(h^2) .
\end{split}
\end{equation*}

Let $p$ be the principal symbol of $-h^2\Delta|_\Gamma$. We notice that from \cite{SZ4}, if $g=g_{\theta_1},\varphi=\varphi(\theta_1)$ are chosen suitably, then $p$ lies in the lower half plane and for every $\delta>0$, there exists $\epsilon>0$ such that
\begin{equation*}
t\geqslant\delta\Rightarrow \epsilon\leqslant-\arg p\leqslant\pi-\epsilon.
\end{equation*}

Now we consider the boundary condition. In the normal geodesic coordinate $(y,t)\in\partial\mathcal{O}\times[0,\infty)$, the Robin boundary condition becomes
\begin{equation*}
\partial_tu+\gamma u|_{t=0}=0
\end{equation*}
Therefore in the scaled operator, we shall choose the boundary condition
\begin{equation*}
e^{-\frac{\pi i}{3}}\partial_tu+\gamma u|_{t=0}=0
\end{equation*}
or
\begin{equation*}
\partial_\nu u+ku|_{\partial\Gamma}=0
\end{equation*}
where $k:\partial\Gamma\to\mathbb{C}$ is a smooth function.

We shall define the scaled operator $P=P^{(\gamma)}_\Gamma=-\Delta|_\Gamma:D(\Gamma)\to L^2(\Gamma)$ where $D(\Gamma)=\{u\in H^2(\Gamma)|\partial_\nu u+ku|_{\partial\Gamma}=0\}$ and regard $P$ as an operator on $\RR^n\setminus\mathcal{O}$ by the parametrization of $\Gamma$ given above. According to \cite{SZ1}, $P$ has discrete spectrum in the sector $e^{-i(0,2\varphi)}(0,+\infty)$. Moreover, we have

\begin{prop}
\label{re-sp}
The resonances of $P^{(\gamma)}$ in $-\varphi<\arg\zeta<0$ is the same as the square root of the eigenvalues of $P$ in $-2\varphi<\arg\zeta<0$.
\end{prop}

The proof of the theorem is based on the following deformation results. First, we have the lemma about non-characteristic deformations which is proved in \cite{SZ1}.

\begin{lem}
\label{nonchd}
Let $\omega\subset\RR^n$ be an open set, $h:[0,1]\times\omega\ni(t,y)\mapsto h(t,y)\in\mathbb{C}^n$ be a smooth proper map such that\\
(1) $\det(\partial_yh(t,y))\neq0$ for all $(t,y)$;\\
(2) $h(t,\cdot)$ is injective;\\
(3) $h(t,y)=h(0,y)$ for $y\in\omega\setminus K$ where $K$ is a compact subset of $\omega$.\\
We write $\Gamma_t=h(\{t\}\times\omega)$. Let $P(x,D_x)$ be a partial differential operator with holomorphic coefficients defined in a neighborhood of $h([0,1]\times\omega)$ such that $P|_{\Gamma_t}$ is elliptic for $0\leqslant t\leqslant1$. If $u_0\in\mathscr{D}'(\Gamma_0)$ and $P_{\Gamma_0}u_0$ extends to a holomorphic function in a neighborhood of $h([0,1]\times\omega)$, then $u_0$ extends to a holomorphic function in a neighborhood of $h([0,1]\times\omega)$.
\end{lem}

Next we need to the following lemma providing that we can apply the deformation all the way to the boundary and it will satisfy the desired boundary properties.

\begin{lem}
\label{detob}
Let $u\in C^\infty(\RR^n\setminus\mathcal{O})$ satisfy $(-\Delta-\lambda^2)^{k_0}u=0,\partial^\alpha u|_{\partial\Omega}=\bar{u}_\alpha\in C^\infty(\partial\Omega)$ in a neighborhood of $x_0$. Then there exists a complex neighborhood $W$ of $x_0$ such that
(1) $u$ extends holomorphically to a function $U$ in a complex open neighborhood of $W\cap\bigcup_{|\theta|\leqslant\theta_0}\Gamma_\theta^0$;\\
(2) $u_\theta=U|_{\Gamma_\theta}$ is smooth up to $\partial\Gamma_\theta=\partial\mathcal{O}$;\\
(3) $(-\Delta|_{\Gamma_\theta}-\lambda^2)^{k_0}u_\theta=0,\partial^\alpha u_\theta|_{\partial\Gamma_\theta}=\bar{u}_\alpha$ in $\Gamma_\theta\cap W$.\\
Moreover, we may replace $\RR^n\setminus\mathcal{O}$ by any fixed $\Gamma_\eta$ with $|\eta|<\theta_0$.
\end{lem}
This lemma was proved in \cite{SZ4}. We recall the proof with more details in the appendix. Combining these deformation results with the method in \cite{SZ1}, we have the proposition.

\subsection{Semiclassical Sobolev spaces}
\label{sob}
In this section, we review some basic facts about the semiclassical Sobolev spaces, especially the estimate of the trace operator which we shall use later.

Let $H_h^s(\RR^n)\subset\mathcal{S}'(\RR^n)$ be the semiclassical Sobolev space of order $s$ with the norm
\begin{equation*}
\|u\|_{H_h^s(\RR^n)}=\|\langle hD\rangle^s u\|_{L^2(\RR^n)}, \langle hD\rangle=(1+(hD)^2)^{\frac{1}{2}}.
\end{equation*}

Let $X$ be a compact smooth manifold, we can choose a finite cover $X_1,\ldots,X_p$ of $X$ where $X_1,\ldots, X_p$ are coordinate charts with local coordinates $x_1,\ldots,x_n$. Then there exists a partition of unity $\chi_j\in C_0^\infty(X_j)$, $\sum_{j=1}^p\chi_j=1$. We define the semiclassical Sobolev space $H_h^s(X)$ to be the space of all $u\in\mathcal{D}'(X)$ such that
\begin{equation*}
\|u\|_{H_h^s(X)}^2=\sum_{j=1}^p\|\chi_j\langle hD\rangle^s\chi_ju\|_{L^2(X_j)}^2<\infty.
\end{equation*}

For different choice of the coordinate charts and partition of unity, the norms are equivalent uniformly for $h>0$. Also, another equivalent norm can be given by $\|u\|_{H_h^s(X)}=\|(I-h^2\Delta)^{\frac{s}{2}}u\|_{L^2(X)}$, where $\Delta$ is the Laplacian operator with respect to some Riemannian metric. From this norm, we see that $H_h^s(X)$ is a Hilbert space with inner product $\langle u,v\rangle_{H_h^s(X)}=\langle(I-h^2\Delta)^{\frac{s}{2}}u,(I-h^2\Delta)^{\frac{s}{2}}v\rangle_{L^2(X)}$.

Now consider the trace operator $\Tr:C^\infty(\RR^n\setminus\mathcal{O})\to C^\infty(\partial\mathcal{O}),u\mapsto u|_{\partial\mathcal{O}}$. In the normal geodesic coordinates given in the previous section, it is equivalent to the operator $\Tr:C^\infty(X\times[0,\infty))\to C^\infty(X),\Tr u(y)=u(0,y)$.

\begin{prop}
\label{trace}
For $u\in C^\infty_0(X\times[0,\infty))$, we have $\|\Tr u\|_{H_h^1(X)}^2\leqslant Ch^{-1}\|u\|_{H_h^2(X\times[0,\infty))}^2$.
\end{prop}

\begin{proof}
Since $u\in C^\infty_0(X\times[0,\infty))$, we know there exists $L>0$ such that $u$ is supported in $X\times[0,L]$. Therefore
\begin{equation*}
\begin{split}
\|\Tr u\|_{H_h^1(X)}^2&=-h^{-1}\int_0^\infty hD_t\|u(\cdot,t)\|_{H_h^1(X)}^2dt\leqslant 2h^{-1}\int_0^\infty|\langle hD_tu(\cdot,t),u(\cdot,t)\rangle_{H_h^1(X)}|dt\\
&\leqslant h^{-1}\int_0^\infty[\|hD_tu(\cdot,t)\|_{H_h^1(X)}^2+\|u(\cdot,t)\|_{H_h^1(X)}^2]dt\leqslant Ch^{-1}\|u\|_{H_h^2(X\times[0,\infty))}^2.
\end{split}
\end{equation*}
\end{proof}

\begin{rem}
In fact, we can even prove that $\Tr=O_s(h^{-\frac{1}{2}}):H_h^s(\Omega)\to H_h^{s-\frac{1}{2}}(\partial\Omega)$ when $s>\frac{1}{2}$, $\Omega\subset\subset\RR^n$ an open set with smooth boundary.
\end{rem}

\subsection{The FBI transform}
\label{fbi}
This section is devoted to a brief introduction of Fourier-Bros-Iagolnitzer (FBI) transform on $\RR^n$ or a smooth compact manifold. The FBI transform was first introduced by Bros and Iagolnitzer to study the analytic singularity of a distribution and later on became an extremely useful tool in microlocal analysis. We refer to the books by Delort \cite{De}, Folland \cite{Fo} and Sj\"{o}strand \cite{Sj} for more details. An account of the semiclassical theory needed here can be found in
Martinez \cite{Ma} and Zworski \cite[Chapter 13]{EZ}. In our argument an important part is
played by the proof of the sharp G{\aa}rding inequality given by Cordoba and Fefferman \cite{CF}.
Wunsch and Zworski \cite{WZ} (see also \cite{SjC} for the analytic case) adapted the FBI transform to compact
 Riemannian manifolds.

We shall first review the basic facts about FBI transform on $\RR^n$, then follow \cite{WZ}
to describe the FBI transform on a compact manifold which for us will eventually be $\partial \mathcal O$.

The FBI transform of a function $u\in\mathscr{S}'(\RR^n)$ is given by
\begin{equation*}
T_hu(x,\xi)=2^{-\frac{n}{2}}(\pi h)^{-\frac{3n}{4}}\int_{\RR^n}e^{-\frac{1}{2h}(x-y)^2+\frac{i}{h}(x-y)\cdot\xi}u(y)dy.
\end{equation*}
This is a special case of the wave packet transform
\begin{equation*}
T_hu(x,\xi)=\langle u,\varphi_{(x,\xi,h)}\rangle_{\mathscr{S}',\mathscr{S}},
\end{equation*}
where $\varphi_{(x,\xi,h)}$ is a ``wave packet'' concentrated at $(x,\xi)\in T^\ast\RR^n$. If we take $\varphi_{(x,\xi,h)}$ to be the coherent state, then we get the FBI transform. The FBI transform is also closely related to the Bargmann transform in several complex variables:
\begin{equation*}
\tilde{T}_hu(z)=\int_{\RR^n}e^{-\frac{1}{2h}(z-y)^2}u(y)dy, z\in\mathbb{C}^n.
\end{equation*}
If we identify $z=x-i\xi\in\mathbb{C}^n$ with $(x,\xi)\in T^\ast\RR^n$, then
\begin{equation*}
T_hu(x,\xi)=2^{-\frac{n}{2}}(\pi h)^{-\frac{3n}{4}}e^{-\frac{1}{2h}\xi^2}\tilde{T}_hu(z).
\end{equation*}
Let us define $L^2_\Phi(\mathbb{C}^n)$ to be the $L^2$-space with the weight $e^{-\frac{1}{h}(\Im z)^2}m(dz)$ where $m(dz)$ is the Lebesgue measure on $\mathbb{C}^n$, then $\tilde{T}_h:L^2(\RR^n)\to L^2_\Phi(\mathbb{C}^n)$ is an isometry with image $H_\Phi(\mathbb{C}^n)=\{f\in L^2_\Phi(\mathbb{C}^n): f \mbox{ is holomorphic in } \mathbb{C}^n\}$.

Back to the FBI transform $T_h:L^2(\RR^n)\to L^2(\RR^{2n})$, the adjoint $T^\ast_h:L^2(\RR^{2n})\to L^2(\RR^n)$ is given by
\begin{equation*}
T_h^\ast v(y)=2^{-\frac{n}{2}}(\pi h)^{-\frac{3n}{4}}\int_{\RR^{2n}}e^{-\frac{1}{2h}(x-y)^2-\frac{i}{h}(x-y)\cdot\xi}v(x,\xi)dxd\xi
\end{equation*}
From the properties of the Bargmann transform, we know that $T_h$ is an isometry with image $L^2(\RR^{2n})\cap e^{-\frac{\xi^2}{2h}}A(\mathbb{C}^n_{x-i\xi})$. Thus $T_h^\ast T_hu=u$ for every $u\in L^2(\RR^n)$ and $T_hT_h^\ast$ is the orthogonal projection in $L^2(\RR^{2n})$ onto the image $L^2(\RR^{2n})\cap e^{-\frac{\xi^2}{2h}}A(\mathbb{C}^n_{x-i\xi})$.\\

Now let $(X,g)$ be a compact Riemannian manifold. We write $y$ to be a point on $X$, $dy$ to be the volume form on $X$; $(x,\xi)$ a point on $T^\ast X$ (where $x\in X,\xi\in T_x^\ast X$) and $dxd\xi$ to be the canonical volume form on $T^\ast X$.

An admissible phase function $\varphi(x,\xi,y)$ is a smooth function on $T^\ast X\times X$ satisfying the following conditions:
\begin{equation*}
\begin{split}
& \text{(1) $\varphi$ is an elliptic polyhomogeneous symbol of order one in $\xi$;}\\
& \text{(2) $\Im\varphi\geqslant0$;}\\
& \text{(3) $d_y\varphi|_\Delta=-\xi dy$;}\\
& \text{(4) $d_y^2\Im\varphi|_\Delta\sim\langle\xi\rangle$;}\\
& \text{(5) $\varphi|_\Delta=0$.}
\end{split}
\end{equation*}

Therefore near the diagonal $\Delta$, $\varphi=\xi\cdot(x-y)+\langle Q(x,\xi,y)(x-y),(x-y)\rangle$ where $Q$ is a symmetric matrix-valued symbol of degree 1 in $\xi$ with $\Im Q|_{\Delta}\sim\langle\xi\rangle I$.

The symbol class $h^mS^k_{phg}(T^\ast X\times X)$ is defined to be the collection of all smooth functions
\begin{equation*}
a=a(x,\xi,y;h)\sim h^m(a_k(x,\xi,y)+ha_{k-1}(x,\xi,y)+\cdots)
\end{equation*}
where $a_j(x,\xi,y)$ are polyhomogeneous symbols of degree $j$ in $\xi$ and the asymptotic expansion means that
\begin{equation*}
|a-h^m(a_k+\cdots+h^ja_{k-j})|\leqslant C_jh^{m+j+1}|\xi|^{k-j-1}, |\xi|>1.
\end{equation*}
The symbol class $h^mS^k_{phg}(T^\ast X)$ is defined similarly, without the $y$-components.

A symbol $a\in h^mS^k_{phg}$ is called elliptic if the principal part $|a_k|\sim\langle\xi\rangle$ uniformly with respect to other variables. The quantization of $a$ is defined as the operator
\begin{equation*}
\Op(a)u(y)=\frac{1}{(2\pi h)^n}\int e^{i\xi\exp_x^{-1}(y)/h}a(x,\xi,y;h)u(x)\chi(y,x)dxd\xi,
\end{equation*}
where $\exp$ is the exponential map with respect to $g$ on $X$. We shall write $h^m\Psi^k(X)$ to be the algebra of all operators $\Op(a)+R,a\in h^mS^k,R=O(h^\infty):C^{-\infty}(X)\to C^\infty(X)$.

Then following \cite{WZ}, an FBI transform on $X$ is an operator $T_h:C^\infty(X)\to C^\infty(T^\ast X)$ given by
\begin{equation}
\label{fbim}
T_hu(x,\xi)=\int_Xe^{\frac{i}{h}\varphi(x,\xi,y)}a(x,\xi,y;h)\chi(x,\xi,y)u(y)dy.
\end{equation}
Here  $\varphi(x,\xi,y)$ is an admissible phase function; $a(x,\xi,y;h)\in h^{-\frac{3n}{4}}S^{\frac{n}{4}}$ is an elliptic polyhomogeneous symbol; $\chi(x,\xi,y)$ is a cut-off function to a small neighborhood of the diagonal $\Delta=\{(x,\xi,y)\in T^\ast X\times X:x=y\}$ such that $\Im\varphi\leqslant-C^{-1}d(x,y)^2$ on the support of $\chi$.

The following properties of the FBI transform were proved in \cite{WZ}:\\
(1) $T_h:L^2(X)\to L^2(T^\ast X)$ is bounded for $h<h_0$;\\
(2) We can choose a suitable phase $\varphi$ and elliptic symbol $a$ such that
\begin{equation}
\label{aliso}
\|T_hu\|_{L^2(T^\ast X)}=(1+O(h^\infty))\|u\|_{L^2(X)},
\end{equation}
i.e. $T_h$ is an isometry modulo $h^\infty$.

From now on, we shall always use such kind of FBI transforms. Furthermore, we know from \cite{WZ}:
\begin{lem}
Let $P=\Op(p)\in h^k\Psi^m(X)$, then $T_h^{\ast}pT_h-P\in h^{k+1}\Psi^{m-1}$.
\end{lem}

We can apply this to prove the following

\begin{prop}
(1) For any $u\in C^\infty(X)$,
\begin{equation}
\label{fbi1}
\|\langle\xi\rangle T_hu\|_{L^2(T^\ast X)}\leqslant C\|u\|_{H_h^1(X)}.
\end{equation}
(2) If $A(x,hD_x)$ is a second-order differential operator on $X$, then for any $u\in C^\infty(X)$,
\begin{equation}
\label{fbi2}
\|A(x,\xi)Tu\|_{L^2(T^\ast X)}^2=\|A(x,hD_x)u\|_{L^2(X)}^2+O(h)\|u\|_{H_h^2}^2.
\end{equation}
\end{prop}

\begin{proof}
(1)
\begin{equation*}
\begin{split}
\|\langle\xi\rangle T_hu\|^2&=\langle\langle\xi\rangle T_hu,\langle\xi\rangle T_hu\rangle=\langle T_h^\ast\langle\xi\rangle^2T_hu,u\rangle\\
&=\langle(I-\Delta)u,u\rangle+\langle Ru,u\rangle=\|u\|_{H_h^1(X)}+\langle Ru,u\rangle,
\end{split}
\end{equation*}
where $R\in h\Psi^1$. So $\langle Ru,u\rangle=O(h)\|u\|_{H_h^{\frac{1}{2}}(X)}^2$.

(2) Notice that by symbol calculus
\begin{equation*}
(\bar{A}A)(x,hD)=A(x,hD)^\ast A(x,hD) \mod{h\Psi^3}
\end{equation*}
we have the following
\begin{equation*}
\begin{split}
\|A(x,\xi)Tu\|_{L^2(T^\ast X)}^2&=\langle A(x,\xi)Tu,A(x,\xi)Tu\rangle=\langle T^\ast|A(x,\xi)|^2Tu,u\rangle\\
&=\langle A(x,hD)^\ast A(x,hD)u,u\rangle+\langle Ru,u\rangle=\|A(x,hD)u\|^2+\langle Ru,u\rangle,
\end{split}
\end{equation*}
where $R\in h\Psi^3$. So $\langle Ru,u\rangle=O(h)\|u\|_{H_h^{\frac{3}{2}}(X)}^2$.
\end{proof}

\begin{rem}
We also notice that all the discussion above work for functions with value in a Hilbert space $\mathscr{H}$. In our case, we shall choose $X=\partial\mathcal{O}$ and $\mathscr{H}=L^2([0,\infty))$.
\end{rem}

\section{Estimates of the Airy-type operator}
\label{airy}

In this section, we shall give the lower bounds for the ordinary differential operator
\begin{equation}
\label{eq:airy}
P=e^{-2\pi i/3}((hD_t)^2+t)+O(h)hD_t+O(h+h^{\frac{1}{2}}t+t^2)
\end{equation}
define on $[0,\infty)$ with general conditions at $t=0$.

\subsection{The Dirichlet and Neumann realization}
\label{dn}

Let $Ai$ be the Airy function defined by
\begin{equation*}
Ai(s)=\frac{1}{2\pi}\int_{\Im\sigma=\delta>0}e^{i(\sigma^3/3)+i\sigma s}d\sigma.
\end{equation*}
Then we can give all the eigenfunctions and eigenvalues for the Dirichlet and Neumann realization of the Airy operator $D_s^2+s$ on $[0,\infty)$:
\begin{equation*}
(D_s^2+s)Ai(s-\zeta_j)=\zeta_jAi(s-\zeta_j),\ \ \ Ai(-\zeta_j)=0;
\end{equation*}
\begin{equation*}
(D_s^2+s)Ai(s-\zeta_j')=\zeta_j'Ai(s-\zeta_j'),\ \ \ Ai'(-\zeta_j')=0,
\end{equation*}
where $0<\zeta_1<\zeta_2<\cdots,0<\zeta_1'<\zeta_2'<\cdots$ and $\zeta_1\thickapprox2.338,\zeta_1'\thickapprox1.019$. The spectral theorem gives the following estimates:

Let $v\in C_0^\infty[0,\infty)$, if $v(0)=0$, then
\begin{equation}
\label{ei:d0}
\langle(D_s^2+s)v,v\rangle\geqslant\zeta_1\|v\|^2;
\end{equation}
if $D_sv(0)=0$, then
\begin{equation}
\label{ei:n0}
\langle(D_s^2+s)v,v\rangle\geqslant\zeta_1'\|v\|^2.
\end{equation}

Now we consider the semiclassical version of the Airy operator, $(hD_t)^2+t$. By changing the variables $t=h^{\frac{2}{3}}s$, we have \begin{equation*}
(hD_t)^2+t=h^{\frac{2}{3}}(D_s^2+s).
\end{equation*}
For $u=u(t)$, we define $v(s)=h^{\frac{1}{3}}u(h^{\frac{2}{3}}s)$, then $u(t)=h^{-\frac{1}{3}}v(h^{-\frac{2}{3}}t)$, $\|u\|_{L^2_t}=\|v\|_{L^2_s}$ and \begin{equation*}
((hD_t)^2+t)u(t)=h^{\frac{1}{3}}(D_s^2+s)v(s).
\end{equation*}
Therefore
\begin{equation*}
\langle((hD_t)^2+t)u,u\rangle_{L^2_t}=h^{\frac{2}{3}}\langle(D_s^2+s)v,v\rangle_{L^2_s}.
\end{equation*}

Applying the estimates \eqref{ei:d0} and \eqref{ei:n0}, we have for $u\in C_0^\infty[0,\infty)$, if $u(0)=0$, then
\begin{equation}
\label{ei:dh0}
\langle((hD_t)^2+t)u,u\rangle\geqslant\zeta_1h^{\frac{2}{3}}\|u\|^2
\end{equation}
if $D_tu(0)=0$, then
\begin{equation}
\label{ei:nh0}
\langle((hD_t)^2+t)u,u\rangle\geqslant\zeta_1'h^{\frac{2}{3}}\|u\|^2
\end{equation}

We also have the following useful identity: for $u\in C_0^\infty([0,\infty)), u(0)=0$ or $D_tu(0)=0$,
\begin{equation*}
\langle((hD_t)^2+t)u,u\rangle=\langle(hD_t)^2u,u\rangle+\langle
tu,u\rangle=\|hD_tu\|^2+\|t^{\frac{1}{2}}u\|^2,
\end{equation*}
and consequently
\begin{equation}
\label{ei:dnh1}
\langle((hD_t)^2+t)u,u\rangle\geqslant\|hD_tu\|^2;
\end{equation}
\begin{equation}
\label{ei:dnht}
\langle((hD_t)^2+t)u,u\rangle\geqslant\|t^{\frac{1}{2}}u\|^2.
\end{equation}

Now we could estimate $\|((hD_t)^2+t)u\|$ by the Cauchy-Schwartz inequality
\begin{equation*}
\|((hD_t)^2+t)u\|\|u\|\geqslant\langle((hD_t)^2+t)u,u\rangle.
\end{equation*}

If $u(0)=0$, then by \eqref{ei:dh0}
\begin{equation}
\label{eo:dh0}
\|((hD_t)^2+t)u\|\geqslant\zeta_1h^{\frac{2}{3}}\|u\|,
\end{equation}
and by \eqref{ei:dnh1}
\begin{equation*}
\|((hD_t)^2+t)u\|^2\geqslant\zeta_1h^{\frac{2}{3}}\langle((hD_t)^2+t)u,u\rangle\geqslant\zeta_1h^{\frac{2}{3}}\|hD_tu\|^2,
\end{equation*}
thus
\begin{equation}
\label{eo:dh1}
\|((hD_t)^2+t)u\|\geqslant\sqrt{\zeta_1}h^{\frac{1}{3}}\|hD_tu\|
\end{equation}

Similarly, if $D_tu(0)=0$, by \eqref{ei:nh0} and \eqref{ei:dnh1}, we have
\begin{equation}
\label{eo:nh0}
\|((hD_t)^2+t)u\|\geqslant\zeta_1'h^{\frac{2}{3}}\|u\|;
\end{equation}
\begin{equation}
\label{eo:nh1}
\|((hD_t)^2+t)u\|\geqslant\sqrt{\zeta_1'}h^{\frac{1}{3}}\|hD_tu\|.
\end{equation}

Another way to estimate $\|((hD_t)^2+t)u\|$ is to use the following identity: If $u(0)=0$ or $D_tu(0)=0$,
since $\langle u,hD_tu\rangle$ is real
\begin{equation*}
\begin{split}
\|((hD_t)^2+t)u\|^2&=\|(hD_t)^2u\|^2+\|tu\|^2+2\Re\langle tu,(hD_t)^2u\rangle\\
                   &=\|(hD_t)^2u\|^2+\|tu\|^2+2\Re\langle hD_t(tu),hD_tu\rangle\\
                   &=\|(hD_t)^2u\|^2+\|tu\|^2+2\Re\langle thD_tu,hD_tu\rangle+h\Re\frac{2}{i}\langle u,hD_tu\rangle\\
                   &=\|(hD_t)^2u\|^2+\|tu\|^2+2\|t^{\frac{1}{2}}hD_tu\|^2.
\end{split}
\end{equation*}

This gives us the following estimates
\begin{equation}
\label{eo:dnh2}
\|((hD_t)^2+t)u\|^2\geqslant\|(hD_t)^2u\|^2+\|tu\|^2\geqslant\|(hD_t)^2u\|^2.
\end{equation}

\subsection{General condition}
\label{gc}
Now we remove the Dirichlet or Neumann condition at $t=0$ and try to get a lower bound of $\langle((hD_t)^2+t)u,u\rangle$. In this case, $(hD_t)^2+t$ is no longer a self-adjoint operator, but the semiclassical setting allows us to view it as a perturbation of the Neumann realization.
We shall start with the following basic estimate:
\begin{equation*}
\begin{split}
\|hD_tu\|^2&=\langle hD_tu,hD_tu\rangle=\langle(hD_t)^2u,u\rangle-ih^2D_tu(0)\bar{u}(0)\\
&\leqslant\langle((hD_t)^2+t)u,u\rangle-ih^2D_tu(0)\bar{u}(0).
\end{split}
\end{equation*}
Since the right hand side is real, we have
\begin{equation*}
\begin{split}
\|hD_tu\|^2&\leqslant\Re\langle((hD_t)^2+t)u,u\rangle-\Re(ih^2D_tu(0)\bar{u}(0))\\
           &\leqslant\Re\langle((hD_t)^2+t)u,u\rangle+h^2|D_tu(0)||u(0)|\\
           &\leqslant\Re\langle((hD_t)^2+t)u,u\rangle+O(h^2)|D_tu(0)|^2+O(h^2)|u(0)|^2,
\end{split}
\end{equation*}
or
\begin{equation}
\label{ei:h1}
\Re\langle((hD_t)^2+t)u,u\rangle\geqslant\|hD_tu\|^2-O(h^2)|D_tu(0)|^2-O(h^2)|u(0)|^2.
\end{equation}
which is the analogue of \eqref{ei:dnh1} for general $u$. Next we try to get an analogue of \eqref{ei:dh0} and \eqref{ei:nh0}.

\begin{lem}
Suppose $u\in C_0^\infty([0,\infty))$, then we have the following estimate:
\begin{equation}
\label{eir:h0}
\Re\langle((hD_t)^2+t)u,u\rangle\geqslant\zeta_1'h^{\frac{2}{3}}(1-O(h^{\frac{2}{3}}))\|u\|^2-O(h^2)|D_tu(0)|^2,
\end{equation}
\begin{equation}
\label{eii:h0}
|\Im\langle((hD_t)^2+t)u,u\rangle|\leqslant O(h^{\frac{2}{3}})\Re\langle((hD_t)^2+t)u,u\rangle+O(h^2)|D_tu(0)|^2.
\end{equation}
\end{lem}

\begin{proof}

Write $D_tu(0)=a$ for simplicity. First, we use the scaling $t=h^{\frac{2}{3}}s$, $v(s)=h^{\frac{1}{3}}u(h^{\frac{2}{3}}s)$ as before. We have
\begin{equation*}
\langle((hD_t)^2+t)u,u\rangle_{L^2_t}=h^{\frac{2}{3}}\langle(D_s^2+s)v,v\rangle_{L^2_s},
\end{equation*}
and more importantly, $D_sv(0)=hD_tu(0)=ha$. Now let $v(s)=w(s)+has\chi(s)$ where $\chi\in C_0^\infty([0,\infty))$ is a fixed function such that $\chi\equiv 1$ near 0. Then we get the decomposition
\begin{equation*}
\begin{split} 
& \langle(D_s^2+s)v,v\rangle=\\ \ 
& \ \ \ \langle(D_s^2+s)w,w\rangle+ha\langle(D_s^2+s)s\chi,w\rangle
+h\bar{a}\langle(D_s^2+s)w,s\chi\rangle+h^2|a|^2\langle(D_s^2+s)s\chi,s\chi\rangle.
\end{split}
\end{equation*}

Since $w$ satisfies the Neumann condition $D_sw(0)=0$, we know the first term is real. We can integrate by parts to rewrite the third term as $-h\bar{a}\langle D_sw,D_s(s\chi)\rangle+h\bar{a}\langle w,s^2\chi\rangle$. Therefore for the real part, by the Cauchy-Schwarz inequality,
\begin{equation*}
\begin{split}
\Re\langle(D_s^2+s)v,v\rangle&\geqslant\langle(D_s^2+s)w,w\rangle-O(h)|a|\|D_sw\|-O(h)|a|\|w\|-O(h^2)|a|^2\\
                             &\geqslant\langle(D_s^2+s)w,w\rangle-O(h^{\frac{4}{3}})|a|^2-O(h^{\frac{2}{3}})(\|w\|^2+\|D_sw\|^2).
\end{split}
\end{equation*}
From
\begin{equation*}
\begin{split}
\|D_sw\|^2&= \langle D_sw,D_sw\rangle=\langle D_s^2w,w\rangle\leqslant\langle(D_s^2+s)w,w\rangle\\
\|w\|^2&\leqslant\zeta_1'^{-1}\langle(D_s^2+s)w,w\rangle\ \ \ \ \ \ (\mbox{by}\eqref{ei:n0}),
\end{split}
\end{equation*}
we have
\begin{equation}
\label{e:vtow}
\Re\langle(D_s^2+s)v,v\rangle\geqslant(1-O(h^{\frac{2}{3}}))\langle(D_s^2+s)w,w\rangle-O(h^{\frac{4}{3}})|a|^2.
\end{equation}

By \eqref{ei:n0} and
\begin{equation*}
\begin{split}
\|w\|^2&\geqslant (\|v\|-\|has\chi\|)^2=\|v\|^2-O(h)|a|\|v\|+O(h^2)|a|^2\\
       &\geqslant (1-O(h^{\frac{2}{3}}))\|v\|^2-O(h^{\frac{4}{3}})|a|^2,
\end{split}
\end{equation*}
we get
\begin{equation*}
\begin{split}
\Re\langle(D_s^2+s)v,v\rangle&\geqslant \zeta_1'(1-O(h^{\frac{2}{3}}))\|w\|^2-O(h^{\frac{4}{3}})|a|^2\\
                             &\geqslant \zeta_1'(1-O(h^{\frac{2}{3}}))\|v\|^2-O(h^{\frac{4}{3}})|a|^2.
\end{split}
\end{equation*}

For the imaginary part, we have
\begin{equation*}
\begin{split}
|\Im\langle(D_s^2+s)v,v\rangle|&\leqslant O(h^{\frac{4}{3}})|a|^2+O(h^{\frac{2}{3}})(\|w\|^2+\|D_sw\|^2)\\
                               &\leqslant O(h^{\frac{4}{3}})|a|^2+O(h^{\frac{2}{3}})\langle(D_s^2+s)w,w\rangle.
\end{split}
\end{equation*}
Using \eqref{e:vtow} again, we have
\begin{equation*}
|\Im\langle(D_s^2+s)v,v\rangle|\leqslant O(h^{\frac{4}{3}})|a|^2+O(h^{\frac{2}{3}})\Re\langle(D_s^2+s)v,v\rangle.
\end{equation*}
Scaling back from $v$ to $u$, we get the desired estimates \eqref{eir:h0} and \eqref{eii:h0}.
\end{proof}

Finally, we need an analogue of \eqref{eo:dnh2}. The argument in Section \ref{dn} shows that
\begin{equation*}
\|((hD_t)^2+t)u\|^2=\|(hD_t)^2u\|^2+\|tu\|^2+2\|t^{\frac{1}{2}}hD_tu\|^2+h\Re\frac{2}{i}\langle u,hD_tu\rangle.
\end{equation*}
Although the last term is no longer zero, we can calculate it through integration by parts. Since
\begin{equation*}
\langle u,hD_tu\rangle=\langle hD_tu,u\rangle-hi|u(0)|^2,
\end{equation*}
we have
\begin{equation*}
\Im\langle u,hD_tu\rangle=-\frac{h}{2}|u(0)|^2.
\end{equation*}
Therefore
\begin{equation*}
\Re\frac{2}{i}\langle u,hD_tu\rangle=-h|u(0)|^2,
\end{equation*}
and we get
\begin{equation}
\label{eo:h2}
\|((hD_t)^2+t)u\|^2\geqslant\|(hD_t)^2u\|^2-h^2|u(0)|^2
\end{equation}

\subsection{Restriction to a small interval}
\label{si}

Now we restrict the support of $u$ to a small fixed interval and get a better estimate.

\begin{lem}
If $L>0$ is sufficiently large, $0<h<h_0(L)$, then the following estimates holds uniformly for $u\in C_0^\infty([0,L^{-1}])$:
\begin{equation}
\label{eir:si}
\Re\langle((hD_t)^2+t)u,u\rangle\geqslant(\zeta_1'h^{\frac{2}{3}}-O(hL))\|u\|^2-O(h^2)|D_tu(0)|^2+\frac{L}{2}\|tu\|^2.
\end{equation}
\end{lem}

\begin{proof} Choose $\chi_0,\chi_1\in C^\infty(\RR)$ such that
\begin{equation*}
\begin{split}
\chi_0^2+\chi_1^2&=1, 0\leqslant\chi_j\leqslant1\\
\chi_0&\equiv1\ \mbox{ on }(-\infty,h^{\frac{1}{2}}]\\
\chi_1&\equiv1\ \mbox{ on }[2h^{\frac{1}{2}},\infty)\\
\partial^\alpha\chi_j&=O_\alpha(h^{-\frac{\alpha}{2}}),\alpha=0,1,2.
\end{split}
\end{equation*}
Then we can deduce that
\begin{equation*}
\chi_0[\chi_0,(hD_t)^2]+\chi_1[\chi_1,(hD_t)^2]=-[\chi_0(hD_t)^2(\chi_0)+\chi_1(hD_t)^2(\chi_1)]=O(h),
\end{equation*}
from which we have
\begin{equation*}
\begin{split}
\langle((hD_t)^2+t)u,u\rangle=&\ \langle\chi_0((hD_t)^2+t)u,\chi_0u\rangle+\langle\chi_1((hD_t)^2+t)u,\chi_1u\rangle\\
                             =&\ \langle((hD_t)^2+t)\chi_0u,\chi_0u\rangle+\langle((hD_t)^2+t)\chi_1u,\chi_1u\rangle\\
                             &-\langle(\chi_0[\chi_0,(hD_t)^2]+\chi_1[\chi_1,(hD_t)^2])u,u\rangle.
\end{split}
\end{equation*}
Since $\chi_1u(0)=0$, also $\chi_0u$ and $u$ have the same condition at $t=0$, we can apply the estimates \eqref{ei:dnht} and \eqref{eir:h0} to get
\begin{equation*}
\Re\langle((hD_t)^2+t)u,u\rangle
\geqslant\zeta_1'h^{\frac{2}{3}}(1-O(h^{\frac{2}{3}}))\|\chi_0u\|^2+\|t^{\frac{1}{2}}\chi_1u\|^2-O(h^2)|D_tu(0)|^2-O(h)\|u\|^2.
\end{equation*}
From the construction of $\chi_0,\chi_1$, it is easy to see
\begin{equation*}
\|u\|^2=\|\chi_0u\|^2+\|\chi_1u\|^2, \|tu\|^2=\|t\chi_0u\|^2+\|t\chi_1u\|^2.
\end{equation*}
Therefore we only need to prove for some $c_0=O(L)$, we have
\begin{equation}
\label{chi0}
c_0h\|\chi_0u\|^2\geqslant\frac{L}{2}\|t\chi_0u\|^2
\end{equation}
and
\begin{equation}
\label{chi1}
\|t^{\frac{1}{2}}\chi_1u\|^2\geqslant(\zeta h^{\frac{2}{3}}-c_0h)\|\chi_1u\|^2+\frac{L}{2}\|t\chi_1u\|^2.
\end{equation}

Since $\chi_0$ is supported on $[-\infty,2h^{\frac{1}{2}}]$, we only need to choose $c_0=2L$ to get \eqref{chi0}. To prove \eqref{chi1}, we need to show that for $t\in[h^{\frac{1}{2}},L^{-1}]$, $t\geqslant\zeta h^{\frac{2}{3}}-c_0h+\frac{L}{2}t^2$ or equivalently,
\begin{equation*}
\frac{L}{2}(t-L^{-1})^2\leqslant\frac{1}{2L}+2hL-\zeta_1h^{\frac{2}{3}}.
\end{equation*}
The left hand side achieves its maximum at $t=h^{\frac{1}{2}}$, so we only need
\begin{equation*}
h^{\frac{1}{2}}\geqslant\zeta h^{\frac{2}{3}}-c_0h+\frac{L}{2}h
\end{equation*}
which can be achieved by choosing $h<h_0(L)$ small enough.
\end{proof}

\subsection{Airy operator with lower order terms}
\label{aol}

Now we shall include the lower order terms and prove the main result of this section.

\begin{prop}
Suppose that a second order ordinary differential operator on $[0,\infty)$ satisfies that
\begin{equation*}
P=e^{-2\pi i/3}((hD_t)^2+t)+O(h)hD_t+O(h+h^{\frac{1}{2}}t+t^2)
\end{equation*}
Let $\omega_0\in\mathbb{C}$ with $\arg\omega_0\in(-\frac{1}{6}\pi,\frac{5}{6}\pi)$. If $L>0$ is sufficiently large, $h>0$ sufficiently small depending on $L$ and $\omega_0$, then for $u\in C_0^\infty([0,L^{-1}))$,
\begin{equation}
\label{lowairy}
\begin{split}
\|(P-\omega_0)u\|^2\geqslant&\ (|e^{2\pi i/3}\omega_0-\zeta_1'h^{\frac{2}{3}}|^2-O(hL))\|u\|^2\\
                            &+C^{-1}\|(hD_t)^2u\|^2-O(h^2)|D_tu(0)|^2-O(h^2)|u(0)|^2.
\end{split}
\end{equation}
\end{prop}

\begin{proof}
We begin with the following identity
\begin{equation*}
\begin{split}
\|(P-\omega_0)u\|^2=&\ \|(e^{-\frac{2\pi i}{3}}((hD_t)^2+t)-\omega_0)u\|^2+\|[O(h)hD_t+O(h+h^{\frac{1}{2}}t+t^2)]u\|^2\\
                   &+2\Re\langle(e^{-\frac{2\pi i}{3}}((hD_t)^2+t)-\omega_0)u,[O(h)hD_t+O(h+h^{\frac{1}{2}}t+t^2)]u\rangle\\
                   \geqslant&\ \|(e^{-\frac{2\pi i}{3}}((hD_t)^2+t)-\omega_0)u\|^2-[O(h)\langle((hD_t)^2+t)u,hD_tu\rangle\\
                   &+O(h)\langle u,hD_tu\rangle+\langle((hD_t)^2+t)u,O(h+t^2)u\rangle+\langle u,O(h+t^2)u\rangle].
\end{split}
\end{equation*}

The lower order terms are estimated as follows,
\begin{equation*}
\begin{split}
O(h)\langle((hD_t)^2+t)u,hD_tu\rangle\leqslant&\ O(h)\|((hD_t)^2+t)u\|^2+O(h)\|hD_tu\|^2\\
                                     \leqslant&\ O(h)\|((hD_t)^2+t)u\|^2+O(h)\Re\langle((hD_t)^2+t)u,u\rangle\\
                                     &+O(h^3)|D_tu(0)|^2+O(h^3)|u(0)|^2  \ \ \text{ (by \eqref{ei:h1})};\\
O(h)\langle u,hD_tu\rangle           \leqslant&\ O(h)\|u\|^2+O(h)\|hD_tu\|^2\\
                                     \leqslant&\ O(h)\|u\|^2+O(h)\Re\langle((hD_t)^2+t)u,u\rangle\\
                                     & +O(h^3)|D_tu(0)|^2+O(h^3)|u(0)|^2 \ \ \text{ (by \eqref{ei:h1})};\\
\langle((hD_t)^2+t)u,O(h+t^2)u\rangle\leqslant&\ O(h)\langle((hD_t)^2+t)u,u\rangle+O(1)\langle((hD_t)^2+t)u,t^2u\rangle\\
                                     \leqslant&\ O(h)\|((hD_t)^2+t)u\|^2+O(h)\|u\|^2\\
                                     & +\textstyle{\frac{1}{2}}\|((hD_t)^2+t)u\|^2+O(1)\|t^2u\|^2,\\
\langle u,O(h+t^2)u\rangle \leqslant&\ O(h)\|u\|^2+O(1)\|tu\|^2.
\end{split}
\end{equation*}

For the leading terms, we use the following identities
\begin{equation*}
\begin{split}
\|(e^{-\frac{2\pi i}{3}}((hD_t)^2+t)-\omega_0)u\|^2=&\ \|((hD_t)^2+t)u\|^2+|\omega_0|^2\|u\|^2\\
&-2\Re\langle e^{-\frac{2\pi i}{3}}((hD_t)^2+t)u,\omega_0u\rangle,
\end{split}
\end{equation*}
and
\begin{equation*}
\begin{split}
-2\Re\langle e^{-\frac{2\pi i}{3}}((hD_t)^2+t)u,\omega_0u\rangle=&\ 2\Re[e^{\frac{\pi i}{3}}\bar{\omega}_0\langle((hD_t)^2+t)u,u\rangle]\\
=&\ \Re(2e^{\frac{\pi i}{3}}\bar{\omega}_0)\cdot\Re\langle((hD_t)^2+t)u,u\rangle\\
 &-\Im(2e^{\frac{\pi i}{3}}\bar{\omega_0})\cdot\Im\langle((hD_t)^2+t)u,u\rangle.
\end{split}
\end{equation*}

By \eqref{eii:h0}, the second term is bounded below by
\begin{equation*}
-2|\omega_0||\Im\langle((hD_t)^2+t)u,u\rangle|\geqslant-O(h^{\frac{2}{3}})\Re\langle((hD_t)^2+t)u,u\rangle
-O(h^2)|D_tu(0)|^2.
\end{equation*}
Therefore
\begin{equation*}
\begin{split}
&-2\Re\langle e^{-\frac{2\pi i}{3}}((hD_t)^2+t)u,\omega_0u\rangle\\
&\ \ \ \ \ \ \geqslant(2\cos(\textstyle{\frac{\pi}{3}}-\arg\omega_0)-O(h^{\frac{2}{3}}))|\omega_0|\Re\langle((hD_t)^2+t)u,u\rangle-O(h^2)|D_tu(0)|^2.
\end{split}
\end{equation*}

Now combining all the terms together, we get the following estimate
\begin{equation*}
\begin{split}
\|(P-\omega_0)u\|^2\geqslant &
(2\cos(\textstyle{\frac{\pi}{3}}-\arg\omega_0)-O(h^{\frac{2}{3}}))|\omega_0|\Re\langle((hD_t)^2+t)u,u\rangle\\
&+(|\omega_0|^2-O(h))\|u\|^2+(\textstyle{\frac{1}{2}}-O(h))\|((hD_t)^2+t)u\|^2\\
&-O(1)\|tu\|^2-O(1)\|t^2u\|^2-O(h^2)|D_tu(0)|^2-O(h^3)|u(0)|^2.
\end{split}
\end{equation*}

Since $|\frac{\pi}{3}-\arg\omega_0|<\frac{\pi}{2}, \cos(\frac{\pi}{3}-\arg\omega_0)>0$, when $h$ is small enough, we can apply \eqref{eir:si} to the first term and use
\begin{equation*}
|\omega_0|^2-2\Re(e^{-\frac{2\pi i}{3}}\bar{\omega_0})\zeta_1h^{\frac{2}{3}}=|e^{2\pi i/3}\omega_0-\zeta_1h^{\frac{2}{3}}|^2-O(h)
\end{equation*}
to get
\begin{equation*}
\begin{split}
\|(P-\omega_0)u\|^2 \geqslant & \left((|e^{\frac{2\pi i}{3}}\omega_0-\zeta_1h^{\frac{2}{3}}|^2-O(hL)\right)\|u\|^2
 +\left(\textstyle{\frac{1}{2}}-O(h) \right)\|((hD_t)^2+t)u\|^2\\
& \ \ +\left(|\omega_0|\cos(\textstyle{\frac{\pi}{3}}-\arg\omega_0)-O(h))L-O(1) \right) \|tu\|^2\\
& \ \ \ \  -O(1)\|t^2u\|^2-O(h^2)|D_tu(0)|^2-O(h^3)|u(0)|^2.
\end{split}
\end{equation*}

Since $u$ is supported in $[0,L^{-1}]$, we have $\|t^2u\|^2\leqslant L^{-2}\|tu\|^2$. So if $L>L_0(\omega_0)$ large enough, $h<h_0(L,\omega_0)$ small enough, we get
\begin{equation*}
\begin{split}
\|(P-\omega_0)u\|^2\geqslant&\ (|e^{2\pi i/3}\omega_0-\zeta_1'h^{\frac{2}{3}}|^2-O(hL))\|u\|^2+\left(\textstyle{\frac{1}{2}}-O(h)\right)\|((hD_t)^2+t)u\|^2\\
                            &+C^{-1}L\|tu\|^2-O(h^2)|D_tu(0)|^2-O(h^3)|u(0)|^2.
\end{split}
\end{equation*}
Applying \eqref{eo:h2}, we conclude the proof of \eqref{lowairy}.
\end{proof}

\begin{rem}
If we replaced $(hD_t)^2+t$ by $(hD_t)^2+Qt$, then the estimates (with $\zeta_1'h^{\frac{2}{3}}$ replaced by $\zeta_1'Q^{\frac{2}{3}}h^{\frac{2}{3}}$) remain uniform for $Q\in[C^{-1},C]$, $|\omega_0|\in[C^{-1},C]$ and $\arg(\omega_0)\in[-\frac{\pi}{6}+\delta,\frac{5\pi}{6}-\delta]$ ($\delta,C>0$)
\end{rem}
\begin{rem}
For the Dirichlet and Neumann realization, we can get the same estimate without the last two lower order terms $-O(h^2)|D_tu(0)|^2-O(h^2)|u(0)|^2$ based on the inequalities in Section \ref{dn}. For Dirichlet realization, we can also improve $\zeta_1'$ to $\zeta_1$.
\end{rem}

\section{Lower bounds on the scaled operator}
\label{lb}

In section \ref{cs}, we defined the scaled operator $P=-\Delta|_\Gamma$ and find the explicit formula in normal coordinates with respect to the boundary: $\Omega\ni x\mapsto(y,t)\in X\times(0,\infty)$. If we freeze $(y,\eta)\in T^\ast X$, then the (semiclassical) symbol of $h^2P$ is given by
\begin{equation}
\label{symbol}
P(y,t;\eta,hD_y)=e^{-\frac{2\pi i}{3}}((hD_t)^2+2tQ(y,\eta))+R(y,\eta)+O(t^2+h)\langle\eta\rangle^2+O(h)hD_t.
\end{equation}
In this section, we first estimate $P(y,t;\eta,hD_t)-\omega_0$ and then through the FBI transform introduced in Section \ref{fbi} to get a lower bound on $h^2P-\omega_0$.

\subsection{Estimate in the glancing region}
\label{gl}

When $\Re\omega_0-R(y,\eta)$ is small, the main part of $P(y,t;\eta,hD_t)-\omega_0$ is given by the Airy-type operator $e^{-\frac{2\pi i}{3}}((hD_t)^2+2tQ(y,\eta))$. We shall apply our results in Section \ref{airy} to get the following estimate. Notice that such $(y,\eta)$ lies in a compact subset of $T^\ast X$.

\begin{lem}
Let $\omega_0\in\mathbb{C}$ with $\Re\omega_0>0,\Im\omega_0=r_0>0$. Suppose $|\Re\omega_0-R(x',\xi')|<c$ where $c$ is sufficiently small, $L$ is large enough and $0<h<h_0(L)$. Then For any $v\in C_0^\infty([0,L^{-1}))$, we have
\begin{equation}
\label{e:gl}
\begin{split}
\|(P(y,t,\eta,hD_t)-\omega_0)v\|^2\geqslant|r_0+2S(\Re\omega_0)^\frac{2}{3}h^{\frac{2}{3}}-O(h)|^2\|v\|^2\\
+C^{-1}\|(hD_t)^2v\|^2-O(h^2)|D_tv(0)|^2-O(h^2)|v(0)|^2
\end{split}
\end{equation}
where $S$ is given by
\begin{equation}
\label{const}
S=2^{-\frac{1}{3}}\cos\left(\frac{\pi}{6}\right)\zeta_1'\left(\min_{y\in X,j=1,\ldots,n-1}K_j(y)\right)^{\frac{2}{3}}.
\end{equation}
$K_j(y)$ are the principal curvatures of $X=\partial\mathcal{O}$ at $y$ and $-\zeta_1'$ is the first zero of the derivative of Airy function.
\end{lem}

\begin{proof}
Let $c=r_0\tan\frac{\pi}{6}$, then since $|\Re\omega_0-R(y,\eta)|<c$, we have $\arg(\omega_0-R(y,\eta))\in[\frac{\pi}{3},\frac{2\pi}{3}]$. It follows immediately from \eqref{lowairy} by replacing $\omega_0$ with $\omega_0-R(y,\eta)$ that
\begin{equation*}
\begin{split}
\|(P(y,t,\eta,hD_t)-\omega_0)v\|^2\geqslant(|\omega_0-R(y,\eta)-e^{-\frac{2\pi i}{3}}\zeta_1'(2Q(y,\eta))^{\frac{2}{3}}h^{\frac{2}{3}}|^2-O(hL))\|v\|^2\\
+C^{-1}\|(hD_t)^2v\|^2-O(h^2)|D_tv(0)|^2-O(h^2)|v(0)|^2.
\end{split}
\end{equation*}
The uniformity of the constants follows from the ellipticity of $Q$ and $R$.

Now we need to find a uniform lower bound for
\begin{equation}
\label{loc}
\begin{split}
|\omega_0-R(y,\eta)-e^{-\frac{2\pi i}{3}}\zeta_1'(2Q(y,\eta))^{\frac{2}{3}}h^{\frac{2}{3}}|^2
=&\ |\Re\omega_0-R(y,\eta)+\sin\left(\frac{\pi}{6}\right)\zeta_1'(2Q(y,\eta))^{\frac{2}{3}}h^{\frac{2}{3}}|^2\\
&+|r_0+\cos\left(\frac{\pi}{6}\right)\zeta_1'(2Q(y,\eta))^{\frac{2}{3}}h^{\frac{2}{3}}|^2
\end{split}
\end{equation}
over $(y,\eta)$ such that $|\Re\omega_0-R(y,\eta)|<c$. The minimum is obtained at $R(y,\eta)=\Re\omega_0+O(h^{\frac{2}{3}})$ and the minimum of $\zeta_1'(2Q(y,\eta))^{\frac{2}{3}}$ under such constraint. Since the principal curvatures are the eigenvalues of the quadratic form $Q(y,\eta)$ with respect to the quadratic form $R(y,\eta)$, we have
\begin{equation*}
Q(y,\eta)\geqslant\left(\min_{y\in X,j=1,\ldots,n-1}K_j(y)\right)R(y,\eta).
\end{equation*}
Thus
\begin{equation*}
\eqref{loc}\geqslant|r_0+2S(\Re\omega_0)^{\frac{2}{3}}h^{\frac{2}{3}}|^2+O(h^{\frac{4}{3}})
\end{equation*}
which completes the proof of \eqref{e:gl}.
\end{proof}

\subsection{Estimate away from the glancing region}
\label{agl}

When $|\Re\omega_0-R(y,\eta)|$ is bounded from below, $2tQ(y,\eta)$ is dominated by $R(y,\eta)-\omega_0$ for small $t$. In this case, we can give a better estimate for $P$ from $e^{-\frac{2\pi i}{3}}(hD_t)^2+R(y,\eta)-\omega_0$.

\begin{lem}
Suppose that $\omega_0\in\mathbb{C}$ with $\Re\omega_0>0,\Im\omega_0=r_0>0$,
$|\Re\omega_0-R(y,\eta)|>c$, then for $L$ large enough, $h$ sufficiently small, and $v\in C_0^\infty([0,L^{-1}))$,
\begin{equation}
\label{e:agl}
\begin{split}
\|(P(y,t,\eta,hD_t)-\omega_0)v\|^2\geqslant&\ (r_0+C^{-1})^2\|v\|^2+C^{-1}(\|(hD_t)^2v\|^2+\langle\eta\rangle^4\|v\|^2)\\
&-O(h^2)\langle\eta\rangle^2|v(0)||D_tv(0)|.
\end{split}
\end{equation}
\end{lem}

\begin{proof}
Since
\begin{equation*}
[P(y,t,\eta,hD_t)-\omega_0]v=[e^{-\frac{2\pi i}{3}}(hD_t)^2+R(y,\eta)-\omega_0]v+[O(t^2+t+h)\langle\eta\rangle^2+O(h)hD_t]v,
\end{equation*}
we have
\begin{equation}
\label{es:nm}
\begin{split}
\|(P(y,t,\eta,hD_t)-\omega_0)v\|^2\geqslant&\ (\|[e^{-\frac{2\pi i}{3}}(hD_t)^2+R(y,\eta)-\omega_0]v\|\\
&-\|[O(t^2+t+h)\langle\eta\rangle^2+O(h)hD_t]v\|)^2.
\end{split}
\end{equation}
Now
\begin{equation}
\label{id:nm}
\begin{split}
\|[e^{-\frac{2\pi i}{3}}(hD_t)^2+R(y,\eta)-\omega_0]v\|^2=&\ \|(hD_t)^2v\|^2+|R(y,\eta)-\omega_0|^2\|v\|^2\\
&+2\Re[e^{-\frac{2\pi i}{3}}(R(y,\eta)-\bar{\omega}_0)\langle(hD_t)^2v,v\rangle],
\end{split}
\end{equation}
where
\begin{equation}
\label{id:rei}
\begin{split}
\Re[e^{-\frac{2\pi i}{3}}(R(y,\eta)-\bar{\omega}_0)\langle(hD_t)^2v,v\rangle]=&\ \Re[e^{-\frac{2\pi i}{3}}(R(y,\eta)-\bar{\omega}_0)]\Re\langle(hD_t)^2v,v\rangle\\
&-\Im[e^{-\frac{2\pi i}{3}}(R(y,\eta)-\bar{\omega}_0)]\Im\langle(hD_t)^2v,v\rangle.
\end{split}
\end{equation}
Notice that
\begin{equation*}
\langle(hD_t)^2v,v\rangle=\|hD_tv\|^2+ih^2D_tv(0)\overline{v(0)}.
\end{equation*}
Therefore
\begin{equation}
\label{id:re}
\Re\langle(hD_t)^2v,v\rangle=\|hD_tv\|^2+\Re(ih^2D_tv(0)\overline{v(0)})\geqslant-h^2|D_tv(0)||v(0)|;
\end{equation}
\begin{equation}
\label{id:im}
|\Im\langle(hD_t)^2v,v\rangle|=|\Im(ih^2D_tv(0)\overline{v(0)})|\leqslant h^2|D_tv(0)||v(0)|.
\end{equation}
We can compute that
\begin{equation*}
\Re[e^{-\frac{2\pi i}{3}}(R(y,\eta)-\bar{\omega_0})]=-\frac{1}{2}(R(y,\eta)-\Re\omega_0)+\frac{\sqrt{3}}{2}r_0.
\end{equation*}
In the identity
\begin{equation*}
\begin{split}
\Re[e^{-\frac{2\pi i}{3}}(R(y,\eta)-\bar{\omega}_0)]\Re\langle(hD_t)^2v,v\rangle
=&-\frac{1}{2}(R(y,\eta)-\Re\omega_0)\Re\langle(hD_t)^2v,v\rangle\\
&+\frac{\sqrt{3}}{2}r_0\Re\langle(hD_t)^2v,v\rangle,
\end{split}
\end{equation*}
we apply \eqref{id:re} to the second term and
\begin{equation*}
|(R(y,\eta)-\Re\omega_0)\Re\langle(hD_t)^2v,v\rangle|\leqslant\frac{1}{2}(\|(hD_t)^2v\|^2+|(R(y,\eta)-\Re\omega_0)|^2\|v\|^2)
\end{equation*}
to the first term, and get
\begin{equation}
\label{es:rei}
\begin{split}
\Re[e^{-\frac{2\pi i}{3}}(R(y,\eta)-\bar{\omega}_0)]\Re\langle(hD_t)^2v,v\rangle
\geqslant&-\frac{1}{4}(\|hD_tv\|^2+|R(y,\eta)-\Re\omega_0|^2\|v\|^2)\\
&-h^2\frac{\sqrt{3}}{2}r_0|D_tv(0)||v(0)|.
\end{split}
\end{equation}
By \eqref{id:im}, we also have
\begin{equation}
\label{es:imi}
-\Im[e^{-\frac{2\pi i}{3}}(R(y,\eta)-\bar{\omega}_0)]\Im\langle(hD_t)^2v,v\rangle
\geqslant-h^2|R(y,\eta)-\omega_0||D_tv(0)||v(0)|.
\end{equation}
Combining \eqref{id:nm},\eqref{id:rei},\eqref{es:rei},\eqref{es:imi} together, we have
\begin{equation}
\label{es:main}
\begin{split}
\|[e^{-\frac{2\pi i}{3}}(hD_t)^2+R(y,\eta)-\omega_0]v\|^2\geqslant&\ \frac{1}{2}\|(hD_t)^2v\|^2+(r_0^2+\frac{1}{2}|R(y,\eta)-\Re\omega_0|^2)\|v\|^2\\
&-h^2(|R(y,\eta)-\omega_0|+\sqrt{3}r_0)|D_tv(0)||v(0)|.
\end{split}
\end{equation}

Now we estimate the remainder terms, since $|R(y,\eta)-\Re\omega_0|>c$ and $R(y,\eta)$ is a quadratic form in $\eta$, we have
\begin{equation}
\label{es:quad}
C^{-1}\langle\eta\rangle^2\leqslant|R(y,\eta)-\Re\omega_0|\leqslant C\langle\eta\rangle^2
\end{equation}
for some constant $C>0$ independent of $y$ and $\eta$. Therefore if $L$ is large enough and $h$ is small enough, we have for $v\in C_0^\infty([0,L^{-1}])$,
\begin{equation*}
\|[O(t^2+t+h)\langle\eta\rangle^2+O(h)hD_t]v\|^2\leqslant C^{-2}|R(y,\eta)-\Re\omega_0|^2\|v\|^2+O(h^2)\|hD_tv\|^2.
\end{equation*}
We can apply
\begin{equation*}
\langle(hD_t)^2v,v\rangle=\|hD_tv\|^2+ih^2D_tv(0)\overline{v(0)}
\end{equation*}
again to get
\begin{equation*}
\|hD_tv\|^2\leqslant|\langle(hD_t)^2v,v\rangle|+h^2|D_tv(0)||v(0)|\leqslant\frac{1}{2}(\|(hD_t)^2v\|^2+\|v\|^2)+h^2|D_tv(0)||v(0)|.
\end{equation*}
Therefore
\begin{equation}
\label{es:lot}
\begin{split}
\|[O(t^2+t+h)\langle\eta\rangle^2+O(h)hD_t]v\|^2\leqslant&\ C^{-2}(|R(y,\eta)-\Re\omega_0|^2\|v\|^2\\
&+\|(hD_t)^2v\|^2)+O(h^4)|D_tv(0)||v(0)|.
\end{split}
\end{equation}

It is easy to prove the following elementary inequality:
\begin{equation}
\label{es:ele}
(\sqrt{a-h^2b}-\sqrt{C^{-2}a+h^4b})^2\geqslant(1-2C^{-1})a-2h^2b
\end{equation}
for $C$ large and $h$ small, independent of $a,b>0$.

Applying \eqref{es:ele} with $a=\frac{1}{2}\|(hD_t)^2v\|^2+(r_0^2+\frac{1}{2}|R(y,\eta)-\Re\omega_0|^2)\|v\|^2$, $b=(|R(y,\eta)-\omega_0|+\frac{\sqrt{3}}{2}r_0)|D_tv(0)||v(0)|$, by \eqref{es:nm},\eqref{es:main} and \eqref{es:lot}, we have
\begin{equation*}
\begin{split}
\|(P(y,t,\eta,hD_t)-\omega_0)v\|^2\geqslant(1-2C^{-1})[\frac{1}{2}\|(hD_t)^2v\|^2+(r_0^2+\frac{1}{2}|R(y,\eta)-\Re\omega_0|^2)\|v\|^2]\\
-O(h^2)(|R(y,\eta)-\omega_0|+\sqrt{3}r_0)|D_tv(0)||v(0)|.
\end{split}
\end{equation*}

Now by our assumption, $|R(y,\eta)-\Re\omega_0|>c=r_0\tan\frac{\pi}{6}$, and \eqref{es:quad}, we proved the lemma.
\end{proof}

\subsection{Lower bounds for the scaled operator near the boundary}
\label{bd}

We first consider
an estimate valid for functions supported in a sufficiently small neighborhood of the boundary.
\begin{prop}
Suppose that $u\in C^\infty(\RR^n\setminus\mathcal{O})$ satisfies $\supp(u)\subset X\times[0,L^{-1})$, and the Robin boundary condition $\partial_\nu u=\gamma u$ for $\gamma\in C^\infty(X,\mathbb{C})$. Then
\begin{equation}
\label{lb:nb}
\|(h^2P-\omega_0)u\|^2\geqslant|r_0+S(\Re\omega_0)^{\frac{2}{3}}h^{\frac{2}{3}}-O(h)|^2\|u\|^2.
\end{equation}
\end{prop}

\begin{proof} Let $T$ is the FBI transform defined in Section \ref{fbi}, by \eqref{aliso}, we have
\begin{equation*}
\|u\|_{L^2(X\times[0,L^{-1}))}^2=\|Tu\|^2_{L^2(T^\ast X\times[0,L^{-1}))}+O(h)\|u\|_{L^2}^2.
\end{equation*}
By \eqref{fbi2}, we have
\begin{equation*}
\|u\|_{H_h^2(X\times[0,L^{-1}))}^2\sim\|\langle\eta\rangle^2Tu\|^2_{L^2(T^\ast X\times[0,L^{-1}))}+\|(hD_t)^2Tu\|^2_{L^2(T^\ast X\times[0,L^{-1}))},
\end{equation*}
also
\begin{equation*}
\|(h^2P-\omega_0)u\|^2_{L^2(X\times[0,L^{-1}))}=\|(P(y,t,\eta,hD_t)-\omega_0)Tu\|^2_{L^2(T^\ast X\times[0,L^{-1}))}+O(h)\|u\|_{H_h^2}^2.
\end{equation*}
Now \eqref{e:gl} shows that if $|R(y,\eta)-\Re\omega_0|<c$,
\begin{equation*}
\begin{split}
\int_0^\infty|(P(y,t,\eta,hD_t)-\omega_0)Tu|^2dt\geqslant(r_0+2S(\Re\omega_0)^{\frac{2}{3}}h^{\frac{2}{3}}-O(h))^2\int_0^\infty|Tu|^2dt\\
+C^{-1}\int_0^\infty|(hD_t)^2Tu|^2dt-O(h^2)|D_tTu(0)|^2-O(h^2)|Tu(0)|^2;
\end{split}
\end{equation*}
and \eqref{e:agl} shows that if $|R(y,\eta)-\Re\omega_0|>c$,
\begin{equation*}
\begin{split}
\int_0^\infty|(P(y,t,\eta,hD_t)-\omega_0)Tu|^2dt\geqslant(r_0+C^{-1})^2\int_0^\infty|Tu|^2dt+C^{-1}\int_0^\infty|(hD_t)^2Tu|^2dt\\
+C^{-1}\langle\eta\rangle^4\int_0^\infty|Tu|^2dt-O(h^2)\langle\eta\rangle^2|D_tTu(0)||Tu(0)|.
\end{split}
\end{equation*}
If we integrate $\int_0^\infty|(P(y,t,\eta,hD_t)-\omega_0)Tu|^2dt$ in $(y,\eta)\in T^\ast X$, we get
\begin{equation}
\label{lb:nbt}
\begin{split}
&\|(P(y,t,\eta,hD_t)-\omega_0)Tu\|^2_{L^2(T^\ast X\times[0,L^{-1}))}\\
& \ \ \ \ \ \ \geqslant(r_0+2S(\Re\omega_0)^{\frac{2}{3}}h^{\frac{2}{3}}-O(h))^2\|Tu\|^2_{L^2(T^\ast X\times[0,L^{-1}))}\\
& \ \ \ \ \ \ +C^{-1}\left(\int_{T^\ast X}\int_0^\infty|(hD_t)^2Tu|^2dt+\int_{|R(y,\eta)-\Re\omega_0|>c}\int_0^\infty\langle\eta\rangle^4|Tu|^2dt\right)\\
& \ \ \ \ \ \ -O(h^2)\|\langle\eta\rangle D_tTu(y,\eta,0)\|_{L^2(T^\ast X)}^2-O(h^2)\|\langle\eta\rangle Tu(y,\eta,0)\|^2_{L^2(T^\ast X)}.
\end{split}
\end{equation}
Here $D_tTu(y,\eta,0)=T(D_tu(\cdot,0))(y,\eta)=T(-ku(\cdot,0))$, so by \eqref{fbi1},
\begin{equation*}
\begin{split}
\|\langle\eta\rangle D_tTu(y,\eta,0)\|_{L^2(T^\ast X)}^2=& \|\langle\eta\rangle T(k(\cdot)u(\cdot,0))\|^2_{L^2(T^\ast X)}\\
\leqslant & C\|k(y)u(y,0)\|^2_{H_h^1(X)}\leqslant C\|k\|_{H^1(X)}^2\|u(y,0)\|^2_{H_h^1(X)};\\
\|\langle\eta\rangle Tu(\cdot,0)\|^2_{L^2(T^\ast X)}\leqslant & C\|u(y,0)\|^2_{H_h^1(X)}.
\end{split}
\end{equation*}

Now we can apply Proposition \ref{trace} to the last two terms in \eqref{lb:nbt} to show that they are bounded by $O(h)\|u\|_{H_h^2}$.

Notice that if $|R(y,\eta)-\Re\omega_0|<c$, then $(y,\eta)$ lies in a compact region of $T^\ast X$,
we have
\begin{equation*}
\int_{|R(y,\eta)-\Re\omega_0|<c}\int_0^\infty\langle\eta\rangle^4|Tu|^2dt\leqslant C\|Tu\|^2_{L^2(T^\ast X\times[0,L^{-1}))},
\end{equation*}
thus
\begin{equation*}
\int_{T^\ast X}\int_0^\infty|(hD_t)^2Tu|^2dt+\int_{|R(y,\eta)-\Re\omega_0|>c}\int_0^\infty\langle\eta\rangle^4|Tu|^2dt\geqslant
\max\{0, C^{-1}\|u\|_{H_h^2}^2-C\|u\|_{L^2}^2\}.
\end{equation*}

Therefore from \eqref{lb:nbt}, we have
\begin{equation*}
\begin{split}
\|(h^2P-\omega_0)u\|^2\geqslant &(r_0+2S(\Re\omega_0)^{\frac{2}{3}}h^{\frac{2}{3}}-O(h))^2\|u\|_{L^2}^2\\
&+C^{-1}\max\{0, C^{-1}\|u\|_{H_h^2}^2-C\|u\|_{L^2}^2\}-O(h)\|u\|_{H_h^2}^2.
\end{split}
\end{equation*}
This concludes the proof of \eqref{lb:nb}.
\end{proof}

\subsection{Lower bounds for the scaled operator}
\label{sc}
The estimate away from the boundary is now combined with elliptic estimates away
from the boundary to give the main estimate of the paper:
\begin{thm}
\label{lb:sc}
There exists some $\epsilon>0$ such that for $\omega_0$ satisfying $\arg\omega_0\in(\epsilon,\frac{\pi}{2}-\epsilon)$, $\Re\omega_0\in(1-\epsilon,1+\epsilon)$, and $u\in C^\infty(\RR^n\setminus\mathcal{O})\cap D(\RR^n\setminus\mathcal{O})$(i.e., satisfying the Robin boundary condition), we have
\begin{equation}
\|(h^2P-\omega_0)u\|^2\geqslant|r_0+2S(\Re\omega_0)^{\frac{2}{3}}h^{\frac{2}{3}}-O(h)|^2\|u\|^2.
\end{equation}
for all sufficiently small $h>0$.
\end{thm}

\begin{proof}
We only need to estimate the part away from the boundary and connect it with \eqref{lb:nb}. Let $\varphi_0,\varphi_1\in C^\infty(\RR^n;[0,1])$ such that $\varphi_0^2+\varphi_1^2=1$, $\supp\varphi_0\subset\{x:d(x)<L^{-1}\},\varphi_1=0$ on $\{x:d(x)>(2L)^{-1}\}$ where $d(x)$ is the distance from $x$ to $\mathcal{O}$. We claim that
\begin{equation}
\label{lb:ab}
\|(h^2P-\omega_0)\varphi_1u\|^2\geqslant(r_0+2S(\Re\omega_0)^{\frac{2}{3}}h^{\frac{2}{3}})^2\|\varphi_1u\|^2.
\end{equation}
In fact, from the argument in Section \ref{cs}, the symbol $p$ of $-\Delta|_\Gamma$ when $d(x)>(2L)^{-1}$ takes its values in $\epsilon<-\mbox{arg}z<\pi-\epsilon$ for some $\epsilon>0$. So by the assumption on $\omega_0$,
\begin{equation*}
\begin{split}
\inf|p-\omega_0|&>\Im(e^{i\epsilon}\omega_0)=|\omega_0|\sin(\epsilon+\arg\omega_0)\\
&>|\omega_0|\sin(\arg\omega_0)+2S(\Re\omega_0)^{\frac{2}{3}}h^{\frac{2}{3}}\\
&=r_0+2S(\Re\omega_0)^{\frac{2}{3}}h^{\frac{2}{3}}.
\end{split}
\end{equation*}
We have the estimate \eqref{lb:ab}.

Since
\begin{equation*}
(h^2P-\omega_0)\varphi_ju=\varphi_j(h^2P-\omega_0)u-[\varphi_j,h^2P]u,
\end{equation*}
we have
\begin{equation*}
\begin{split}
\|(h^2P-\omega_0)\varphi_ju\|^2=&\ \|\varphi_j(h^2P-\omega_0)u-[\varphi_j,h^2P]u\|^2\\
\leqslant&\ \|\varphi_j(h^2P-\omega_0)u\|^2+2\|\varphi_j(h^2P-\omega_0)u\|\|[\varphi_j,h^2P]u\|\\
&+\|[\varphi_j,h^2P]u\|^2,
\end{split}
\end{equation*}
thus
\begin{equation*}
\begin{split}
\|(h^2P-\omega_0)u\|^2=&\sum\limits_{j=0,1}\|\varphi_j(h^2P-\omega_0)u\|^2\\
\geqslant&\sum\limits_{j=0,1}\|(h^2P-\omega_0)\varphi_ju\|^2-\sum\limits_{j=0,1}\|[\varphi_j,h^2P]u\|^2\\
&-\sum\limits_{j=0,1}\|\varphi_j(h^2P-\omega_0)u\|\|[\varphi_j,h^2P]u\|.
\end{split}
\end{equation*}

The commutators can be estimated by
\begin{equation*}
\|[\varphi_j,h^2P]u\|=O(h)(\|hD_xu\|_{L^2(X\times[(2L)^{-1},L^{-1}])}+\|u\|)\leqslant O(h)(\|(h^2P-\omega_0)u\|+\|u\|)
\end{equation*}
since $h^2P-\omega_0$ is elliptic when $d(x)>(2L)^{-1}$. Therefore
\begin{equation*}
\|(h^2P-\omega_0)u\|^2\geqslant\sum_{j=0,1}\|(h^2P-\omega_0)\varphi_ju\|^2-O(h)(\|(h^2P-\omega_0)u\|^2+\|u\|^2).
\end{equation*}
Now we can conclude that
\begin{equation*}
\begin{split}
\|(h^2P-\omega_0)u\|^2&\geqslant(1-O(h))\sum\limits_{j=0,1}\|(h^2P-\omega_0)\varphi_ju\|^2-O(h)\|u\|^2\\
&\geqslant(1-O(h))\sum\limits_{j=0,1}|r_0+2S(\Re\omega_0)^{\frac{2}{3}}h^{\frac{2}{3}}-O(h)|^2\|\varphi_ju\|^2-O(h)\|u\|^2\\
&\geqslant|r_0+2S(\Re\omega_0)^{\frac{2}{3}}h^{\frac{2}{3}}-O(h)|^2\|u\|^2.
\end{split}
\end{equation*}
\end{proof}

\section{The pole-free region}
\label{pf}
Now we prove Theorem \ref{polefree}. An equivalent formulation is to say that there are no resonances for $P^{(\gamma)}$ in the region $\Re\zeta>c_0,0<-\mbox{Im}\zeta<S(\Re\zeta)^{\frac{1}{3}}-c_1$ for some constant $c_0,c_1>0$. Suppose $\zeta$ is a resonance of $P^{(\gamma)}$ such that $0<-\mbox{Im}\zeta<S(\Re\zeta)^{\frac{1}{3}}-c_1$. Then by Proposition \ref{re-sp}, $\lambda=\zeta^2$ is an eigenvalue of $P$. Let $h=(\Re\zeta)^{-1}$, then $h^2\zeta^2$ is an eigenvalue of $h^2P$: $h^2Pu=h^2\zeta^2u$ for some $u\in D(\RR^n\setminus\mathcal{O})$. Now we apply  Theorem \ref{lb:sc} to $u$ and $\omega_0=\Re(h^2\zeta^2)+ir_0,r_0>0$. Since $h\Re\zeta=1$, we have $$\Re(h^2\zeta^2)=h^2(\Re\zeta)^2-h^2(\Im\zeta)^2=1+O(h^{\frac{4}{3}});$$
$$\Im(h^2\zeta^2)=2h^2(\Re\zeta)(\Im\zeta)=2h\Im\zeta=O(h^{\frac{4}{3}}).$$
So it is easy to choose some $r_0$ such that $\omega_0$ satisfies the condition in Theorem \ref{lb:sc}(e.g. $r_0=1$), we get
$$|h^2\zeta^2-\omega_0^2|\|u\|_{L^2}^2
=\|(h^2P-\omega_0)u\|_{L^2}^2\geqslant|r_0+2S(\Re\omega_0)^{\frac{2}{3}}h^{\frac{2}{3}}-O(h)|^2\|u\|_{L^2}^2,$$
where $$h^2\zeta^2-\omega_0=i(\Im(h^2\zeta^2)-r_0)=i(2h\Im\zeta-r_0);$$
$$\Re\omega_0=\Re(h^2\zeta^2)=1-h^2(\Im\zeta)^2.$$
Thus we have
$$|r_0-2h\mbox{Im}\zeta|^2\geqslant|r_0+2S[1-h^2(\mbox{Im}\zeta)^2]^{\frac{2}{3}}h^{\frac{2}{3}}-O(h)|^2.$$
so
$$-\mbox{Im}\zeta\geqslant S[1-h^2(\mbox{Im}\zeta)^2]^{\frac{2}{3}}h^{-\frac{1}{3}}-M.$$
Now by the assumption that
$-\mbox{Im}\zeta<Sh^{-\frac{1}{3}}-c_1$, we can choose $c_1$ large, (e.g. $c_1\geqslant S+M$), so that we have
$$1-h^{\frac{1}{3}}\geqslant(1-h^2(\mbox{Im}\zeta)^2)^{\frac{2}{3}}=(1-O(h^{\frac{4}{3}}))^{\frac{2}{3}}.$$
Let $h\to0$, we have a contradiction. Therefore we can choose $c_0,c_1>0$ large such that there are no poles in
$\Re\zeta>c_0, 0<-\mbox{Im}\zeta<S(\Re\zeta)^{\frac{1}{3}}-c_1$.

\section{Appendix}
Now we present the proof of Lemma \ref{detob} following \cite{SZ4}. We need to prove the following\\
(1) $u$ extends holomorphically to a function $U$ in a complex open neighborhood of $W\cap\bigcup_{|\theta|\leqslant\theta_0}\Gamma_\theta^0$;\\
(2) $u_\theta=U|_{\Gamma_\theta}$ is smooth up to $\partial\Gamma_\theta=\partial\mathcal{O}$;\\
(3) $(-\Delta|_{\Gamma_\theta}-\lambda^2)^{k_0}u_\theta=0,\partial^\alpha u_\theta|_{\partial\Gamma_\theta}=\bar{u}_\alpha$ in $\Gamma_\theta\cap W$.

Part (1) and the first equation in (3) follows from Lemma \ref{nonchd} by choosing intermediate contours between $\Gamma_\theta$ and $\RR^n\setminus\mathcal{O}$ away from the boundary. The difficulty lies in justification of the boundary condition for which we need to estimate the norm of $u_\theta$. To do this, we first review the strong uniqueness property of the scaled operator and its corollary. For details, see the appendix of \cite{SZ4}.

\begin{prop}
Assume $P$ is an $m$-th order differential operator with holomorphic coefficients, $\Gamma$ a totally real submanifold of $\mathbb{C}^n$ of maximal dimension such that $P|_\Gamma$ is elliptic. Then if $u\in\mathscr{D}'(\Gamma)$ satisfies $P_\Gamma u=0$ on $\Gamma$ and $u=0$ in a neighborhood of some $x_0\in\Gamma$, then $u\equiv0$.
\end{prop}
\begin{cor}
Let $P,\Gamma$ be as in the proposition above, $\Omega_1\subset\subset\Omega_2\subset\subset\Omega_3\subset\Gamma$ are open sets, then there exists some constant $C>0$ such that for all $u\in H^m(\Omega_3)$,
\begin{equation}
\label{control}
\|u\|_{H^m(\Omega_2)}\leqslant C(\|Pu\|_{H^0(\Omega_3)}+\|u\|_{H^0(\Omega_3\setminus\Omega_1)}).
\end{equation}
Also if $P,\Gamma,\Omega_1,\Omega_2,\Omega_3$ depends continuously on some parameters varying in some compact set, then we have the estimate for some constant $C$ independent of the parameters.
\end{cor}
This corollary shows that if $Pu=0$, then the part of $u$ in $\Omega_3\setminus\Omega_1$ controls the whole part of $u$ in $\Omega_2$. We shall apply this property to a family of intermediate contours between $\Gamma_\theta$ and $\RR^n\setminus\mathcal{O}$ and use the part in $\RR^n\setminus\mathcal{O}$ to control the part in $\Gamma_\theta$.

Now we describe our family of intermediate contours. To do this, we first blow up a neighborhood of $x_0$ by introducing the following change of variables:
\begin{equation*}
x\mapsto\tilde{x}, x=y+\epsilon\tilde{x}
\end{equation*}
where $y\in\partial\mathcal{O}$ is some boundary point near $x_0$, $\epsilon>0$ is a parameter which we will let tend to 0. We shall choose $\tilde{x}_0$ such that $|\tilde{x}_0|=1$ is close to the normal direction to boundary through $y$ and focus on the region $B(\tilde{x}_0,1)$ in the new coordinates.

The intermediate contours are constructed as follows: Let $\Gamma_{\theta,y,\epsilon}$ be the image of $\Gamma_\theta$ in the complexified $\tilde{x}$-space. Since $\Gamma_\theta$ is parametrized by $z=x+i\theta f'(x)$, we have the following parametrization of $\Gamma_{\theta,y,\epsilon}$:
\begin{equation*}
\tilde{z}=\tilde{x}+i\theta\epsilon^{-1}f'(y+\epsilon\tilde{x})=\tilde{x}+i\theta\partial_{\tilde{x}}f_{\epsilon,y}(\tilde{x})
\end{equation*}
where $f_{\epsilon,y}(\tilde{x})=\epsilon^{-2}f(y+\epsilon\tilde{x})$. The derivatives of $f_{\epsilon,y}$ can be estimated as follows:
\begin{equation*}
\partial_{\tilde{x}}f_{\epsilon,y}(\tilde{x})
=\left\{\begin{array}{ccl}
O(|\tilde{x}|^{2-|\alpha|}) & \mbox{if} & |\alpha|\leqslant2\\
O(\epsilon^{|\alpha|-2}) & \mbox{if} & |\alpha|\geqslant2\\
\end{array}\right..
\end{equation*}
We choose a cut-off function $\chi\in C_0^\infty(B(\tilde{x}_0,\frac{1}{2})),0\leqslant\chi\leqslant1$ and $\chi\equiv1$ on $B(\tilde{x}_0,\frac{1}{4})$. Out intermediate contours will be the image of
\begin{equation*}
\tilde{x}\mapsto\tilde{z}=\tilde{x}+i\theta\partial_{\tilde{x}}(\chi f_{\epsilon,y}(\tilde{x}))
\end{equation*}
and we will write $\Omega_0,\Omega_1,\Omega_2,\Omega_3$ to be the images of the balls $B(\tilde{x}_0,\frac{1}{4})$, $B(\tilde{x}_0,\frac{1}{2})$, $B(\tilde{x}_0,\frac{5}{8})$, $B(\tilde{x}_0,\frac{3}{4})$, respectively. See Figure 2(where we omit $\Omega_2$).

\begin{figure}
\label{Fig:cs}
\centering
\includegraphics[height=105mm, width=146mm]{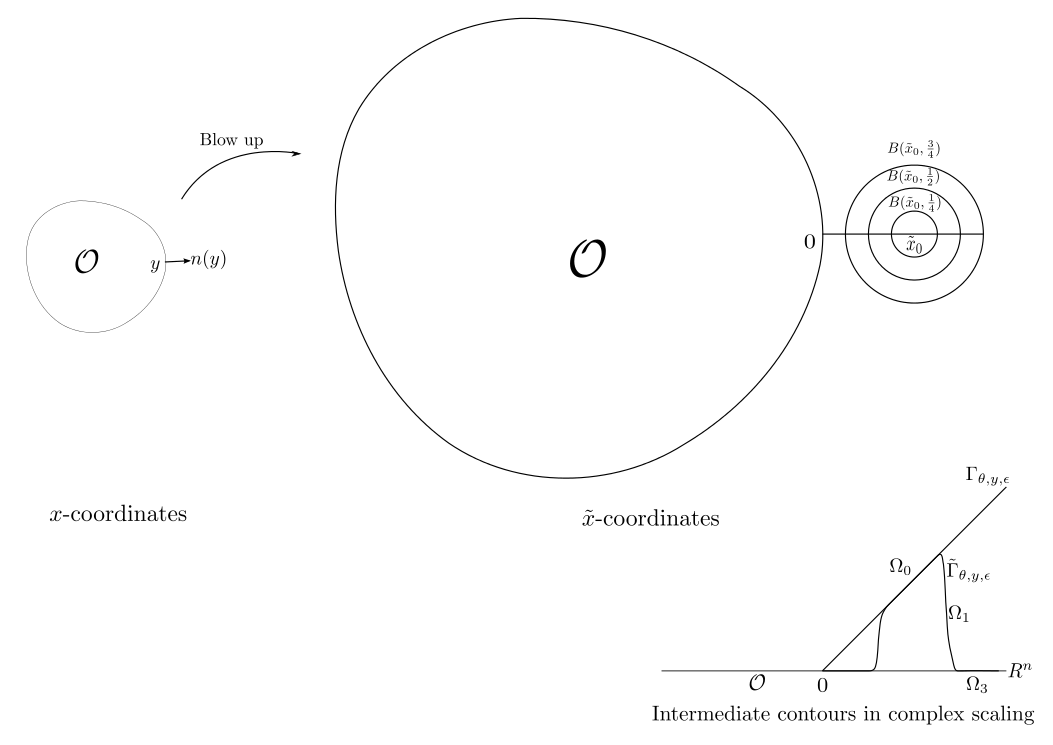}
\caption{Blow up and complex scaling contours. In the complex scaling picture, $\tilde{\Gamma}_{\theta,y,\epsilon}$ is the intermediate contour between $\mathbb{R}^n$ and $\Gamma_{\theta,y,\epsilon}$. On this contour, $\Omega_0,\Omega_1,\Omega_3$ are the images of the balls $B(\tilde{x}_0,\frac{1}{4})$, $B(\tilde{x}_0,\frac{1}{2})$, $B(\tilde{x}_0,\frac{3}{4})$, respectively. 
Therefore $\Omega_0\subset\Gamma_{\theta,y,\epsilon},\Omega_3\setminus\Omega_1\subset\mathbb{R}^n$.}
\end{figure}

By the strong uniqueness property of $(-\Delta_{\tilde{z}}-\epsilon^2\lambda^2)^{k_0}$ and \eqref{control}, we have the following estimates uniformly with respect to $y$ in a neighborhood of $x_0$ and $\epsilon$ small,
\begin{equation*}
\|v\|_{H^{2k_0}(\Omega_2)}\leqslant C(\|(-\Delta_{\tilde{z}}-\epsilon^2\lambda^2)^{k_0}\|_{H^0(\Omega_3)}+\|v\|_{H^0(\Omega_3\setminus\Omega_1)}).
\end{equation*}
We shall only use the following weak version:
\begin{equation}
\label{wcon}
\|v\|_{L^2(\Omega_1)}\leqslant C(\|(-\Delta_{\tilde{z}}-\epsilon^2\lambda^2)^{k_0}\|_{L^2(\Omega_3)}+\|v\|_{L^2(\Omega_3\setminus\Omega_1)}).
\end{equation}

Since in the $\tilde{x}$ coordinates, $\tilde{u}(\tilde{x})=u(x)$ satisfies the equation $(-\Delta_{\tilde{x}}-\epsilon^2\lambda^2)^{k_0}\tilde{u}=0$. By Lemma \ref{nonchd}, $\tilde{u}$ extends to a holomorphic solution $\tilde{U}$ of $(-\Delta_{\tilde{z}}-\epsilon^2\lambda^2)^{k_0}\tilde{U}=0$ over a neighborhood of a family of intermediate contours between $B(\tilde{x}_0,\frac{3}{4})$ and $\Gamma_{\theta,y,\epsilon}^0$ including $\Omega_2\subset\tilde{\Gamma}_{\theta,y,\epsilon}$. Back to the $x$-coordinates, we get the holomorphic extension $U$ of $u$ in $W\cap\Gamma_\theta^0$ if we let $y$ varies near $x_0$ and $\epsilon$ goes to 0.

Substitute $v=\tilde{U}_{\Omega_3}$ in \eqref{wcon}, noticing that $\Omega_3\setminus\Omega_1\subset\RR^n_{\tilde{x}}$, so $\tilde{U}=\tilde{u}$ on $\Omega_3\setminus\Omega_1$, we have
\begin{equation*}
\|\tilde{U}\|_{L^2(\Omega_1)}\leqslant C\|\tilde{u}\|_{L^2(\Omega_3\setminus\Omega_1)}.
\end{equation*}
Back to $x$-coordinates, we will have similar estimate for $U$ and $u$ with the same constant $C$ uniformly for all $y$ near $x_0$ and $\epsilon>0$. In particular, this shows that $u_\theta=U|_{\Gamma_\theta}$ is well-defined in $L^2$ near $x_0$. Since $u_\theta$ satisfies the non-characteristic equation $(-\Delta_z-\lambda^2)u_\theta=0$, if we identify $\Gamma_\theta$ with $\RR^n\setminus\mathcal{O}$ and use the normal geodesic coordinate $(x',x_n)$ (again, only locally near $x_0$), then $u_\theta\in C([0,\epsilon_0);\mathscr{D}'(\RR^{n-1}))$. In particular, it has a boundary value $u_\theta(x',0)\in\mathscr{D}'(\RR^{n-1})$. (For the proof of this, see e.g. \cite{Ho}.)

Now it only remains to show that $u_\theta(x',0)$ coincides with the original boundary value $\bar{u}(x')$. To do this, we substitute $v=\tilde{U}-\tilde{u}(0)$ in \eqref{wcon}, noticing that $\tilde{u}(0)=\bar{u}(y)$,
\begin{equation*}
(-\Delta_{\tilde{z}}-\epsilon^2\lambda^2)^{k_0}(\tilde{U}-\tilde{u}(0))=-(-\epsilon^2\lambda^2)^{k_0}\tilde{u}(0),
\end{equation*}
we have
\begin{equation*}
\|\tilde{U}-\tilde{u}(0)\|_{L^2(\Omega_1)}\leqslant C((\epsilon^2\lambda)^{k_0}|\tilde{u}(0)|\Vol(\Omega_3)^{\frac{1}{2}}+\|\tilde{u}-\tilde{u}(0)\|_{L^2(\Omega_3\setminus\Omega_1)}).
\end{equation*}
Since $\tilde{u}$ is smooth, the last term tends to 0 when $\epsilon\to0$. Therefore
\begin{equation*}
\Vol(\Omega_0)^{-1}\|\tilde{U}-\tilde{u}(0)\|^2_{L^2(\Omega_1)}=o(1), \epsilon\to0.
\end{equation*}
Back to $x$-space and use the normal geodesic coordinates $(x',x_n)$, we have
\begin{equation}
\label{cvtb}
\epsilon^{-n}\|u_\theta(x)-\bar{u}(x')\|^2_{B((x',\epsilon),\frac{\epsilon}{4})}=o(1), \epsilon\to0
\end{equation}
uniformly in $x'$. Now let $\chi_n\in C_0^\infty(\RR),\chi'(x')\in C_0^\infty(\RR^{n-1})$ be cut-off functions with support close to 1 and 0, respectively. Also let $\int\chi_n(x_n)dx_n=\int\chi'(x')dx'=1$. Then for any test function $\varphi\in C_0^\infty(\RR^{n-1})$,
\[ \begin{split}
& \langle u_\theta(x',0)-\bar{u}(x'),\varphi(x')\rangle=\lim_{\epsilon\to0}(u_\theta(x)-\bar{u}(x'))\epsilon^{-1}\chi_n(\epsilon^{-1}x_n)\varphi(x')dx\\
&  \ \ \ \
= \, \lim_{\epsilon\to0}\int\epsilon^{-n}\int(u_\theta(x)-\bar{u}(x'))\chi_n(\epsilon^{-1}x_n)
\chi'(\epsilon^{-1}(x'-y'))\varphi(x')dxdy'=0,
\end{split} \]
since $\epsilon^{-n}\int(u_\theta(x)-\bar{u}(x'))\chi_n(\epsilon^{-1}x_n)\chi'(\epsilon^{-1}(x'-y'))\varphi(x')dx$ has a uniformly compact support with respect to $y'$ and tends to zero uniformly in $y'$ by Cauchy-Schwarz inequality and \eqref{cvtb}.

To get the desired global deformation with suitable boundary condition, we only need to glue all the local deformation together using the strong uniqueness property.

For higher order derivatives, we can repeat the argument for every $\partial^\alpha u$ which satisfies the differential equation $(-\Delta-\lambda^2)^{k_0}(\partial^\alpha u)=0$ to get the holomorphic extension of $\partial^\alpha u$. The strong uniqueness property shows that this is exactly the derivative of the holomorphic extension of $u$.

\end{document}